\documentclass[12pt]{article}
\usepackage[a4paper,margin=1in]{geometry}
\usepackage{amssymb}
\usepackage{amsmath}
\usepackage{mathtools}
\usepackage{mathabx}
\usepackage{amsfonts}
\usepackage{amsthm}
\usepackage{color}
\usepackage{float}
\usepackage[all]{xy}
\usepackage{scalerel}
\usepackage{delimset}

\newcommand{\C}{\mathbb{C}}

\newcommand{\N}{\mathbb{N}}

\newcommand{\E}{\mathrm{E}}

\newcommand{\K}{\mathrm{K}}

\newcommand{\M}{\widehat{\mathrm{M}}}

\newcommand{\D}{\mathbf{D}}

\newtheorem{theorem}{Theorem}
\newtheorem{lemma}{Lemma}
\newtheorem{definition}{Definition}
\newtheorem{proposition}{Proposition}
\newtheorem{corollary}{Corollary}
\newtheorem{example}{Example}

\begin{document}

\title{\textbf{A $q$-Exponential Operator Based on the Derivative of Order 1 and Summation of Bilateral Basic Hypergeometric Series}}
\author{Ronald Orozco L\'opez}

\newcommand{\Addresses}{{
  \bigskip
  \footnotesize

  \textit{E-mail address}, R.~Orozco: \texttt{rj.orozco@uniandes.edu.co}
  
}}

\maketitle

\begin{abstract}
We use a new $q$-exponential operator based on the $q^{\pm1}$-derivative $\D_{q^{\pm1}}$ of order 1 to derive summation formulas for bilateral basic hypergeometric series ${}_{0}\psi_{1}$, ${}_{1}\psi_{1}$, ${}_{1}\psi_{2}$, and ${}_{2}\psi_{2}$. In addition, we provide summation formulas for bilateral series whose terms are basic hypergeometric functions.
\end{abstract}


\section{Introduction}
We begin with some notation and terminology for basic hypergeometric series \cite{gasper}. Let $\vert q\vert<1$. The $q$-shifted factorial is defined by
\begin{align*}
    (a;q)_{n}&=\prod_{k=0}^{n-1}(1-q^{k}a),\\
    (a;q)_{\infty}&=\lim_{n\rightarrow\infty}(a;q)_{n}=\prod_{k=0}^{\infty}(1-aq^{k})
\end{align*}
and the multiple $q$-shifted factorial is defined by
\begin{align*}
    (a_{1},a_{2},\ldots,a_{m};q)_{n}&=(a_{1};q)_{n}(a_{2};q)_{n}\cdots(a_{m};q)_{n},\\
    (a_{1},a_{2},\ldots,a_{m};q)_{n}&=(a_{1};q)_{\infty}(a_{2};q)_{\infty}\cdots(a_{m};q)_{\infty}.
\end{align*}
For $r\leq s$ and $\vert q\vert<1$, a bilateral basic hypergeometric series is defined as
\begin{equation}
    {}_{r}\psi_{s}\left[
    \begin{array}{c}
         a_{1},\ldots,a_{r} \\
         b_{1},\ldots,b_{s}
    \end{array}
    ;q,x
    \right]=\sum_{n=-\infty}^{\infty}\big[(-1)^nq^{\binom{n}{2}}\big]^{s-r}\frac{(a_{1},\ldots,a_{r};q)_{n}}{(b_{1},\ldots,b_{s};q)_{n}}x^n
\end{equation}
which is convergent for $\vert b_{1}\cdots b_{s}/a_{1}\cdots a_{r}\vert<\vert x\vert<1$.
The Ramanujan's summation formula \cite{ramanujan} is
\begin{equation}\label{eqn_ramanujan}
    {}_{1}\psi_{1}(a;b;q,z)=\sum_{n=-\infty}^{\infty}\frac{(a;q)_{n}}{(b;q)_{n}}z^n=\frac{(q,b/a,az,q/az;q)_{\infty}}{(b,q/a,z,b/az;q)_{\infty}},\ \ \vert b/a\vert<\vert z\vert<1.
\end{equation}
The ${}_{r}\phi_{s}$ basic hypergeometric series is define by
\begin{equation*}
    {}_{r}\phi_{s}\left(
    \begin{array}{c}
         a_{1},a_{2},\ldots,a_{r} \\
         b_{1},\ldots,b_{s}
    \end{array}
    ;q,z
    \right)=\sum_{n=0}^{\infty}\frac{(a_{1},a_{2},\ldots,a_{r};q)_{n}}{(q,b_{1},b_{2},\ldots,b_{s};q)_{n}}[(-1)^{n}q^{\binom{n}{2}}]^{1+s-r}z^n,\ \vert z\vert<1.
\end{equation*}
Ramanujan's summation formula, Eq.(\ref{eqn_ramanujan}), is considered as the "bilateral extension" of the $q$-binomial theorem:
\begin{equation}
    {}_1\phi_{0}(a;q,z)=\frac{(az;q)_{\infty}}{(z;q)_{\infty}}=\sum_{n=0}^{\infty}\frac{(a;q)_{n}}{(q;q)_{n}}z^{n}.
\end{equation}
In this paper, for $b\geq0$ we introduce the $q$-operator
\begin{equation}
    \E_{q}(y\D_{q^{\pm}}|q^b)=\sum_{n=0}^{\infty}q^{b\binom{n}{2}}\frac{(y\D_{q^{\pm}})^n}{(q;q)_{n}}
\end{equation}
based on the $q^{\pm1}$-derivative operator $\D_{q^{\pm1}}$ of order 1 defined by
\begin{equation}
    \D_{q^{\pm1}}f(x)=\frac{f(q^{\pm1}x)}{x}.
\end{equation}
The main advantage of the operator $\D_{q^{\pm1}}$ is that it preserves products. In this paper, we derive new summation formulas for bilateral basic hypergeometric series by the method of parameter augmentation based on $\D_{q^{\pm1}}$. This summation formulas involving the Ramanujan's summation formula, Eq.(\ref{eqn_ramanujan}), and the Jacobi theta function 
\begin{equation}
    \vartheta(x;q)=\sum_{n=-\infty}^{\infty}q^{\binom{n+1}{2}}x^n.
\end{equation}
This method based on the $q$-derivative 
\begin{equation*}
    D_{q}f(x)=\frac{f(x)-f(qx)}{(1-q)x}
\end{equation*}
has been used in \cite{somas} to derive summation formulas for other basic bilateral hypergeometric series. 

In addition, we will find several summation formulas for the series ${}_{0}\psi_{1}$, ${}_{1}\psi_{1}$, ${}_{1}\psi_{2}$, and ${}_{2}\psi_{2}$. We provide summation formulas for bilateral series of the form
\begin{equation}
    \sum_{n=-\infty}^{\infty}q^{\binom{n+1}{2}}(dx)^n{}_{r}\phi_{s}\left(
    \begin{array}{c}
         A_{1}(x),\ldots,A_{r}(x) \\
         B_{1}(x),\ldots,B_{s}(x)
    \end{array}
    ;q,A(x,y)
    \right),
\end{equation}
where the $A_{i}(x)=\frac{qa_{i}}{x}$ and $B_{j}(x)=\frac{qb_{j}}{x}$ and $A(x,y)$ is a rational function in the variables $x,y$. Along with the Bailey-Daum summation formula, we obtain news summation formulas for ${}_{2}\psi_{2}$ series. The following easily verified identities will be frequently used in this paper:
\begin{align}
    (q^{n}a;q)_{\infty}&=\frac{(a;q)_{\infty}}{(a;q)_{n}},\label{eqn1}\\
    (q^{-n}a;q)_{\infty}&=\frac{(-a)^n}{q^{\binom{n+1}{2}}}(q/a;q)_{n}(a;q)_{\infty},\label{eqn2}\\
    (a;q^{-1})_{n}&=q^{-\binom{n}{2}}(-a)^{n}(a^{-1};q)_{n}\label{eqn3}
\end{align}

\section{Calculus of order 1}\label{sec3}

\subsection{The $\lambda$-derivative of order 1 of a function}

Define the $\lambda$-shifted operator $\M_{\lambda}$ acting on $f(x)$ to be $\M_{\lambda}\{f(x)\}=f(\lambda x)$.
 
\begin{definition}\label{def1}
For $\lambda\in\C$, we define the $\lambda$-derivative of order 1 $\mathbf{D}_{\lambda}$ of the function $f(x)$ as
\begin{equation}
    (\mathbf{D}_{\lambda}f)(x)=\frac{1}{x}\M_{\lambda}\{f(x)\}=
    \begin{cases}
    \frac{f(\lambda x)}{x},&\text{ if }x\neq0;\\
    \lim_{x\rightarrow0}\frac{f(\lambda x)}{x},&\text{ if }x=0,
    \end{cases}
\end{equation}    
provided that the limit exists.
\end{definition}

It is straightforward to prove the following properties of the $\lambda$-derivative $\mathbf{D}_{q}$.
\begin{theorem}\label{theo_properties}
For all $\alpha,\beta,\gamma\in\C$,
\begin{enumerate}
    \item $\mathbf{D}_{\lambda}\{\alpha f+\beta g\}=\alpha\mathbf{D}_{\lambda}f+\beta\mathbf{D}_{\lambda}g$.
    \item $\mathbf{D}_{\lambda}\{\gamma\}=\frac{\gamma}{x}$.
    \item $\mathbf{D}_{\lambda}\{x^{n}\}=\lambda^{n}x^{n-1}$, for $n\in\N$.
\end{enumerate}
\end{theorem}
\begin{proposition}[{\bf Product $\lambda$-rule}]\label{prop_prod_rule}
\begin{equation}
 \mathbf{D}_{\lambda}\{f(x)g(x)\}=f(\lambda x)\mathbf{D}_{\lambda}g(x)=\mathbf{D}_{\lambda}f(x)\cdot g(\lambda x)=x\mathbf{D}_{\lambda}f(x)\mathbf{D}_{\lambda}g(x).   
\end{equation}
\end{proposition}
From the above proposition, we have the following result.
\begin{proposition}\label{prop2}
    \begin{equation}
        \mathbf{D}_{\lambda}\{f(x)g(x)\}=\frac{1}{2}(f(\lambda x)\mathbf{D}_{\lambda}g(x)+\mathbf{D}_{\lambda}f(x)\cdot g(\lambda x))
    \end{equation}
\end{proposition}

\begin{proposition}[{\bf Quotient $\lambda$-rule}]\label{prop3}
\begin{equation}
 \mathbf{D}_{\lambda}\left\{\frac{f(x)}{g(x)}\right\}=\frac{\mathbf{D}_{\lambda}f(x)}{g(\lambda x)}=\frac{\mathbf{D}_{\lambda}f(x)}{x\mathbf{D}_{\lambda}g(x)}, 
\end{equation}
$g(x)\neq0$, $g(\lambda x)\neq0$.
\end{proposition}

\begin{proposition}\label{prop4}
For all $n\in\N$
    \begin{equation}
        \mathbf{D}_{\lambda}^{n}f(x)=\frac{f(\lambda^nx)}{q^{\binom{n}{2}}x^n},
    \end{equation}
where $\D_{\lambda}^{n}=\D_{\lambda}\D_{\lambda}^{n-1}$    .
\end{proposition}
\begin{proof}
The proof is by induction on $n$. For $n=1$, use the definition $\mathbf{D}_{\lambda}$. Now suppose it is true for $n$. We will prove for $n+1$. We have
\begin{align*}
    \mathbf{D}_{\lambda}^{n+1}f(x)=\mathbf{D}_{\lambda}\mathbf{D}_{\lambda}^{n}f(x)=\mathbf{D}_{\lambda}\left\{\frac{f(\lambda^nx)}{\lambda^{\binom{n}{2}}x^{n}}\right\}=\frac{f(\lambda^{n}q x)}{\lambda^{\binom{n}{2}}\lambda^nx^nx}=\frac{f(\lambda^{n+1}x)}{\lambda^{\binom{n+1}{2}}x^{n+1}}
\end{align*}
\end{proof}

\begin{example}
For all $\alpha\in\C$ and all $k\geq1$,
\begin{equation}
    \D_{\lambda}^kx^\alpha=\lambda^{k\alpha-\binom{k}{2}}x^{\alpha-k}.
\end{equation}
\end{example}

\begin{proposition}[{\bf Leibniz $\lambda$-rule}]\label{prop_leibniz}
For all $n\in\N$,
\begin{equation}
    \mathbf{D}_{\lambda}^{n}(fg)=\lambda^{\binom{n}{2}}x^{n}\mathbf{D}_{\lambda}^{n}(f)\mathbf{D}_{\lambda}^{n}(g).
\end{equation}    
\end{proposition}
\begin{proof}
From Propositions \ref{theo_properties} and \ref{prop4}
\begin{align*}
    \mathbf{D}_{\lambda}^{n}(fg)&=\frac{f(\lambda^{n}x)g(\lambda^{n}x)}{\lambda^{\binom{n}{2}}x^n}\\
    &=\lambda^{\binom{n}{2}}x^n\frac{f(\lambda^nx)}{\lambda^{\binom{n}{2}}x^n}\frac{g(\lambda^nx)}{\lambda^{\binom{n}{2}}x^n}\\
    &=\lambda^{\binom{n}{2}}x^{n}\mathbf{D}_{\lambda}^n(f)\mathbf{D}_{\lambda}^{n}(g).
\end{align*}    
\end{proof}

\begin{proposition}\label{prop6}
For all $n\in\N$,
\begin{equation}
    \D_{\lambda}^n(f_{1}f_{2}\cdots f_{k})=\lambda^{(k-1)\binom{n}{2}}x^{(k-1)n}\D_{\lambda}^n(f_{1})\D_{\lambda}^n(f_{2})\cdots\D_{\lambda}^n(f_{k}).
\end{equation}
\end{proposition}

\subsection{Some $q^{\pm1}$-derivative identities}
In this section, we will find some identities involving $q^{\pm1}$-derivatives. These results will be applied in later sections. The following lemma will be used to find the sum of the following series:
\begin{align*}
    &\sum_{n=-\infty}^{\infty}q^{\binom{n+1}{2}}(cx)^n\cdot{}_{1}\phi_{b-2}\left(
    \begin{array}{c}
         qx/a \\
         \mathbf{0}_{b-2}
    \end{array}
    ;q, (-1)^{b-1}q^{n-1}ay/x^2
    \right),\ b\geq2
\end{align*}
and
\begin{align*}
    \sum_{n=-\infty}^{\infty}q^{\binom{n+1}{2}}(cx)^n{}_{1}\phi_{b+1}\left(
    \begin{array}{c}
         0 \\
         a/x,\mathbf{0}_{b}
    \end{array}
    ;q, (-1)^{b+1}ay/q^{n}x
    \right),\ b\geq0,
\end{align*}
where
\begin{equation}
    \mathbf{0}_{n}=
    \begin{cases}
        (\overbrace{0,\ldots,0)}^n,&\text{ if }n>0;\\
        \emptyset,&\text{ if }n=0.
    \end{cases}
\end{equation}

\begin{lemma}\label{lemma1}
For all $k\in\N$,
\begin{enumerate}
    \item
    \begin{equation}
        \D_{q}^k\{x^n(a/x;q)_{\infty}\}=x^{n-2k}\frac{(-a)^k}{q^{k^2-nk}}(qx/a;q)_{k}(a/x;q)_{\infty}.
    \end{equation}
    \item 
    \begin{equation}
        \D_{q^{-1}}^k\{x^n(a/x;q)_{\infty}\}=x^{n-k}q^{\binom{k}{2}-nk}\frac{(a/x;q)_{\infty}}{(a/x;q)_{k}}.
    \end{equation}
\end{enumerate}
\end{lemma}
\begin{proof}
From Definition \ref{def1} and Eq.(\ref{eqn2})
    \begin{align*}
        \D_{q}^k\{x^n(a/x;q)_{\infty}\}&=(q^{k}x)^n\frac{1}{q^{\binom{k}{2}}x^k}(aq^{-k}/x;q)_{\infty}\\
        &=q^{nk}x^n\frac{1}{q^{\binom{k}{2}}x^k}\frac{(-a/x)^k}{q^{\binom{k+1}{2}}}(qx/a;q)_{k}(a/x;q)_{\infty}\\
        &=x^{n-2k}\frac{(-aq^{n})^k}{q^{k^2}}(qx/a;q)_{k}(a/x;q)_{\infty}.
    \end{align*}
Equally, from Definition \ref{def1} and Eq.(\ref{eqn1})
    \begin{align*}
        \D_{q^{-1}}^{k}\{x^n(a/x;q)_{\infty}\}&=(q^{-k}x)^n\frac{q^{\binom{k}{2}}}{x^k}(aq^{k}/x;q)_{\infty}\\
        &=x^n\frac{q^{\binom{k}{2}-nk}}{x^k}\frac{(a/x;q)_{\infty}}{(a/x;q)_{k}}.
    \end{align*}
\end{proof}
Similarly, the following lemma will be very useful for finding the sum of the following series:
\begin{align*}
    &\sum_{n=-\infty}^{\infty}q^{\binom{n+1}{2}}(cx)^n{}_{1}\phi_{c-1}\left(
    \begin{array}{c}
         x \\
         qax/b,\mathbf{0}_{c-2}
    \end{array}
    ;q,(-1)^{c-1}q^{n+1}y
    \right),\ c\geq2
\end{align*}
and
\begin{align*}
    \sum_{n=-\infty}^{\infty}q^{\binom{n+1}{2}}(dx)^n\cdot{}_{1}\phi_{c+1}\left(
    \begin{array}{c}
         b/ax \\
         q/x,\mathbf{0}_{c}
    \end{array}
    ;q,(-1)^{c+1}aq^{n}y/x
    \right),\ c\geq0.
\end{align*}
\begin{lemma}\label{lemma2}
For all $n\in\N$,
    \begin{equation}
        \D_{q}^n\left\{\frac{(ax,q/ax;q)_{\infty}}{(x,b/ax;q)_{\infty}}\right\}=\frac{q^n}{q^{\binom{n}{2}}}\frac{(x;q)_{n}}{(qax/b;q)_{n}}\frac{(ax,q/ax;q)_{\infty}}{(x,b/ax;q)_{\infty}}
    \end{equation}
and
    \begin{equation}
        \D_{q^{-1}}^n\left\{\frac{(adx,q/adx;q)_{\infty}}{(dx,b/adx;q)_{\infty}}\right\}=\frac{q^{\binom{n}{2}}}{x^n}\frac{a^n(b/ax;q)_{n}}{(q/x;q)_{n}}\frac{(ax,q/ax;q)_{\infty}}{(x,b/ax;q)_{\infty}}.
    \end{equation}
\end{lemma}
\begin{proof}
From Definition \ref{def1} and by simultaneously applying Eqs. (\ref{eqn1}) and (\ref{eqn2}) we obtain that
    \begin{align*}
        \D_{q}^n\left\{\frac{(ax,q/ax;q)_{\infty}}{(x,b/ax;q)_{\infty}}\right\}&=\frac{1}{q^{\binom{n}{2}}x^n}\frac{(q^nax,q^{1-n}/ax;q)_{\infty}}{(q^nx,q^{-n}b/ax;q)_{\infty}}\\
        &=\frac{1}{q^{\binom{n}{2}}x^n}\frac{(x;q)_{n}(-1)^n(ax;q)_{n}(ax)^nq^{\binom{n+1}{2}}}{a^nq^{\binom{n}{2}}(ax;q)_{n}(-1)^n(qax/b;q)_{n}}\frac{(ax,q/ax;q)_{\infty}}{(x,b/ax;q)_{\infty}}\\
        &=\frac{q^n}{q^{\binom{n}{2}}}\frac{(x;q)_{n}}{(qax/b;q)_{n}}\frac{(ax,q/ax;q)_{\infty}}{(x,b/ax;q)_{\infty}}
    \end{align*}
and
    \begin{align*}
        \D_{q^{-1}}^n\left\{\frac{(ax,q/ax;q)_{\infty}}{(x,b/ax;q)_{\infty}}\right\}&=\frac{q^{\binom{n}{2}}}{x^n}\frac{(q^{-n}ax,q^{1+n}/ax;q)_{\infty}}{(q^{-n}x,q^{n}b/ax;q)_{\infty}}\\
        &=\frac{q^{\binom{n}{2}}}{x^n}\frac{(-ax)^n(q/ax;q)_{n}q^{\binom{n+1}{2}}(b/ax;q)_{n}}{q^{\binom{n+1}{2}}(q/ax;q)_{n}(-x)^n(q/x;q)_{n}}\frac{(ax,q/ax;q)_{\infty}}{(x,b/ax;q)_{\infty}}\\
        &=\frac{q^{\binom{n}{2}}}{x^n}\frac{a^n(b/ax;q)_{n}}{(q/x;q)_{n}}\frac{(ax,q/ax;q)_{\infty}}{(x,b/ax;q)_{\infty}}.
    \end{align*}
\end{proof}
Now, denote 
\begin{equation}
    \bigg[
    \begin{array}{c}
         a_{1},\ldots,a_{r}  \\
         b_{1},\ldots,b_{s} 
    \end{array}
    ;q,x
    \bigg]_{n}=\frac{(a_{1}x,a_{2}x,\ldots,a_{r}x;q)_{n}}{(b_{1}x,b_{2}x,\ldots,b_{s}x;q)_{n}}
\end{equation}
and
\begin{equation}
    \bigg[
    \begin{array}{c}
         a_{1},\ldots,a_{r}  \\
         b_{1},\ldots,b_{s} 
    \end{array}
    ;q,x
    \bigg]_{\infty}=\frac{(a_{1}x,a_{2}x,\ldots,a_{r}x;q)_{\infty}}{(b_{1}x,b_{2}x,\ldots,b_{s}x;q)_{\infty}}.
\end{equation}
Finally, to find the sum of the following series
\begin{align*}
    &\sum_{n=-\infty}^{\infty}q^{\binom{n+1}{2}}(dx)^n{}_{r+1}\phi_{r+c-1}\left(
    \begin{array}{c}
         b_{1}x,\ldots,b_{r+1}x, \\
         a_{1}x,\ldots,a_{r+c-1}x
    \end{array}
    ;q,(-1)^{c-1}q^ny/x
    \right)
    \end{align*}
and
\begin{align*}
    \sum_{n=-\infty}^{\infty}q^{\binom{n+1}{2}}(dx)^n{}_{r}\phi_{s+c}\left(
    \begin{array}{c}
         q/a_{1}x,\ldots,q/a_{r}x \\
         q/b_{1}x,\ldots,q/b_{s}x,\mathbf{0}
    \end{array}
    ;q,(-1)^{c+1}\frac{q^{s-r-n}ya_{1}\cdots a_{r}}{x^{1+s-r}b_{1}\cdots b_{s}}
    \right)
\end{align*}
we will apply the following lemma.
\begin{lemma}\label{lemma3}
For all $k\in\N$ we have that
    \begin{align}
        \D_{q}^{k}\left\{x^n\bigg[
    \begin{array}{c}
         a_{1},\ldots,a_{r}  \\
         b_{1},\ldots,b_{s} 
    \end{array}
    ;q,x
    \bigg]_{\infty}\right\}=\frac{x^{n-k}}{q^{\binom{k}{2}-kn}}\bigg[
    \begin{array}{c}
         b_{1},\ldots,b_{s}  \\
         a_{1},\ldots,a_{r} 
    \end{array}
    ;q,x
    \bigg]_{k}\bigg[
    \begin{array}{c}
         a_{1},\ldots,a_{r}  \\
         b_{1},\ldots,b_{s} 
    \end{array}
    ;q,x
    \bigg]_{\infty}
    \end{align}
and    
    \begin{align}
        \D_{q^{-1}}^{k}\left\{x^n\bigg[
    \begin{array}{c}
         a_{1},\ldots,a_{r}  \\
         b_{1},\ldots,b_{s} 
    \end{array}
    ;q,x
    \bigg]_{\infty}\right\}&=x^n\big[(-1)^kq^{\binom{k}{2}}\big]^{(1+s-r)}\left(-q^{s-r}x^{r-s-1}\frac{a_{1}\cdots a_{r}}{q^nb_{1}\cdots b_{s}}\right)^k\nonumber\\
        &\hspace{0.5cm}\times\bigg[
    \begin{array}{c}
         q/a_{1},\ldots,q/a_{r}  \\
         q/b_{1},\ldots,q/b_{s} 
    \end{array}
    ;q,1/x
    \bigg]_{k}\bigg[
    \begin{array}{c}
         a_{1},\ldots,a_{r}  \\
         b_{1},\ldots,b_{s} 
    \end{array}
    ;q,x
    \bigg]_{\infty}
    \end{align}
\end{lemma}
\begin{proof}
The above statements are deduced from Proposition \ref{prop4} and Eq.(\ref{eqn1}) as shown below
    \begin{align*}
        &\D_{q}^{k}\left\{x^n\bigg[
    \begin{array}{c}
         a_{1},\ldots,a_{r}  \\
         b_{1},\ldots,b_{s} 
    \end{array}
    ;q,x
    \bigg]_{\infty}\right\}\\
    &=\frac{(q^kx)^n}{q^{\binom{k}{2}}x^k}\frac{(a_{1}q^kx,\cdots,a_{r}q^{k}x;q)_{\infty}}{(b_{1}q^{k}x,\cdots,b_{s}q^kx;q)_{\infty}}=\frac{x^{n-k}}{q^{\binom{k}{2}-kn}}\frac{(b_{1},\ldots,b_{s};q)_{k}}{(a_{1},\ldots,a_{r};q)_{k}}\frac{(a_{1},\ldots,a_{r};q)_{\infty}}{(b_{1},\ldots,b_{s};q)_{\infty}}
    \end{align*}
and from Eq.(\ref{eqn2})
    \begin{align*}
        &\D_{q^{-1}}^{k}\left\{x^n\bigg[
    \begin{array}{c}
         a_{1},\ldots,a_{r}  \\
         b_{1},\ldots,b_{s} 
    \end{array}
    ;q,x
    \bigg]_{\infty}\right\}\\
    &=(q^{-k}x)^{n}\frac{q^{\binom{k}{2}}}{x^k}\frac{(a_{1}q^{-k}x,\ldots,a_{r}q^{-k}x;q)_{\infty}}{(b_{1}q^{-k}x,\ldots,b_{s}q^{-k}x;q)_{\infty}}\\
        &=x^n\big[(-1)^kq^{\binom{k}{2}}\big]^{(1+s-r)}\left(-q^{s-r}x^{r-s-1}\frac{a_{1}\cdots a_{r}}{q^{n}b_{1}\cdots b_{s}}\right)^k\\
        &\hspace{4cm}\times\bigg[
    \begin{array}{c}
         q/a_{1},\ldots,q/a_{r}  \\
         q/b_{1},\ldots,q/b_{s} 
    \end{array}
    ;q,1/x
    \bigg]_{k}\bigg[
    \begin{array}{c}
         a_{1},\ldots,a_{r}  \\
         b_{1},\ldots,b_{s} 
    \end{array}
    ;q,x
    \bigg]_{\infty}.
    \end{align*}
\end{proof}
Some specialization from Lemma \ref{lemma3} are
\begin{corollary}
    \begin{align}
        \D_{q^{-1}}^{k}\left\{x^n(ax;q)_{\infty}\right\}&=x^n\left(-\frac{a}{q^{n+1}}\right)^k
        (q/ax;q)_{k}(ax;q)_{\infty}.
    \end{align}
\end{corollary}

\begin{corollary}
    \begin{align}
        \D_{q^{-1}}^{k}\left\{x^n\frac{1}{(ax;q)_{\infty}}\right\}&=
        \frac{x^{n-2k}q^{2\binom{k}{2}-(n-1)k}}{a^k(q/ax;q)_{k}(ax;q)_{\infty}}.
    \end{align}
\end{corollary}

\begin{corollary}
    \begin{align}
        \D_{q^{-1}}^{k}\left\{x^n\frac{(ax;q)_{\infty}}{(bx;q)_{\infty}}\right\}&=q^{\binom{k}{2}-nk}\left(\frac{a}{b}\right)^kx^{n-k}
        \frac{(q/ax;q)_{k}}{(q/bx;q)_{k}}\frac{(ax;q)_{\infty}}{(bx;q)_{\infty}}.
    \end{align}
\end{corollary}

\subsection{Derivative of order 1 of the Jacobi theta function}

For a nonzero complex number $x$, the Jacobi theta function is defined as
\begin{align}
    \vartheta(x;q)&=\sum_{n=-\infty}^{\infty}q^{\binom{n+1}{2}}x^{n}.
\end{align}
The Jacobi theta function has the following product representation 
\begin{align}\label{theta_prod}
    \vartheta(x;q)&=(q;q)_{\infty}(-qx;q)_{\infty}(-x^{-1};q)_{\infty}.
\end{align}

\begin{proposition}\label{propo_thetaqn}
For all $n\in\N$
    \begin{align}
        \D_{q}^n\vartheta(ax;q)&=\frac{\vartheta(q^{-n}ax;q)}{q^{\binom{n}{2}}x^n}=\frac{\vartheta(ax;q)}{(ax^2)^nq^{n^2}}
    \end{align}
and
\begin{equation}
    \D_{q^{-1}}^n\vartheta(ax;q)=q^{\binom{n}{2}}\frac{\vartheta(q^nax;q)}{x^n}=a^n\vartheta(ax;q).
\end{equation}
\end{proposition}
\begin{proof}
We apply Proposition \ref{prop4} to Eq.(\ref{theta_prod})
    \begin{align*}
        \D_{q}^n\vartheta(ax;q)&=\frac{(q;q)_{\infty}}{q^{\binom{n}{2}}x^n}(-q^{1+n}ax;q)_{\infty}(-q^{-n}/ax;q)_{\infty}.
    \end{align*}
Now, we use simultaneously Eqs. (\ref{eqn1}) and (\ref{eqn2})    
\begin{align*}
        \D_{q}^n\vartheta(ax;q)&=\frac{(q;q)_{\infty}}{q^{\binom{n}{2}}x^n}\frac{(-qax;q)_{\infty}}{(-qax;q)_{n}}\frac{(-qax;q)_{n}(-1/ax;q)_{\infty}}{(ax)^nq^{\binom{n+1}{2}}}\\
        &=\frac{\vartheta(ax;q)}{(ax^2)^nq^{n^2}}.
    \end{align*}
Similarly, we obtain    
    \begin{align*}
        \D_{q^{-1}}^n\vartheta(ax;q)&=\frac{q^{\binom{n}{2}}(q;q)_{\infty}}{x^n}(-q^{1-n}ax;q)_{\infty}(-q^{n}a^{-1}x^{-1};q)_{\infty}\\
        &=\frac{q^{\binom{n}{2}}(q;q)_{\infty}}{x^n}\frac{(-ax;q^{-1})_{n}}{(-a^{-1}x^{-1};q)_{n}}(-qax;q)_{\infty}(-a^{-1}x^{-1};q)_{\infty}\\
        &=\frac{q^{\binom{n}{2}}}{x^n}\vartheta(ax;q)\prod_{k=0}^{n-1}\frac{1+axq^{-k}}{1+a^{-1}x^{-1}q^{k}}\\
        &=\frac{q^{\binom{n}{2}}}{x^n}\vartheta(ax;q)\prod_{k=0}^{n-1}\frac{ax}{q^{k}}\\
        &=\frac{q^{\binom{n}{2}}}{x^n}\frac{a^nx^n}{q^{\binom{n}{2}}}\vartheta(ax;q)=a^n\vartheta(ax;q).
    \end{align*}
\end{proof}

\section{The $q$-operator $\E_{q}(y\D_{q^{\pm1}}|q^b)$}

\begin{definition}\label{def_PTO}
For all $b\geq0$ define the $q$-operator $\E_{q}(y\D_{q^{\pm1}}|q^b)$ based on the $q^{\pm}$-derivative $\D_{q^{\pm}}$ by letting 
\begin{equation*}
    \E_{q}(y\D_{q^{\pm1}}|q^b)=\sum_{n=0}^{\infty}q^{b\binom{n}{2}}\frac{y^n}{(q;q)_{n}}\D_{q^{\pm}}^n,
\end{equation*}
where $\D_{q^{\pm}}$ refers to any of the derivatives of order 1 $\D_{q}$ and $\D_{q^{-1}}$.
\end{definition}
It is natural to define the following function
\begin{equation}
    \E_{b}(y;q)=\sum_{n=0}^{\infty}q^{b\binom{n}{2}}\frac{y^n}{(q;q)_{n}}
\end{equation}
for $b\geq0$. If $b=0,1,2$, we obtain respectively
\begin{align*}
    \E_{0}(y;q)&=\frac{1}{(y;q)_{\infty}},\\
    \E_{1}(y;q)&=(-y;q)_{\infty},\\
    \E_{2}(y;q)&=\K_{\infty}(y)=\sum_{n=0}^{\infty}q^{n(n-1)}\frac{y^n}{(q;q)_{n}}.
\end{align*}

\begin{theorem}
For all $b\in\N$
\begin{equation}\label{eqn_ope_powx}
    E_{q}(y\D_{q^{\pm}}|q^b)\{x^{n}\}=x^n\E_{b\mp1}(q^{\pm n}y/x;q).
\end{equation}
Some specializations of Eq.(\ref{eqn_ope_powx}) for the derivative $\D_{q}$ are
\begin{align}
    \E_{q}(y\D_{q}|q)\{x^n\}&=x^n\frac{(y/x;q)_{n}}{(y/x;q)_{\infty}},\\
    \E_{q}(y\D_{q}|q^2)\{x^n\}&=x^n\frac{(-y/x;q)_{\infty}}{(-y/x;q)_{n}},\\
    \E_{q}(y\D_{q}|q^3)\{x^n\}&=x^n\K_{\infty}(q^ny/x;q),
\end{align}
and for the derivative $\D_{q^{-1}}$ are
    \begin{align}
        \E_{q}(y\D_{q^{-1}}|1)\{x^{n}\}&=q^{-\binom{n+1}{2}}y^n(-qx/y;q)_{n}(-y/x;q)_{\infty}.\\
        \E_{q}(y\D_{q^{-1}}|q)\{x^{n}\}&=x^n\K_{\infty}(q^{-n}y/x).
    \end{align}
\end{theorem}
\begin{proof}
By applying the Definition \ref{def_PTO}, we get 
    \begin{align*}
        \E_{q}(y\D_{q^{\pm}}|q^b)\{x^{n}\}
        &=\sum_{k=0}^{\infty}q^{b\binom{n}{2}}\frac{y^k}{(q;q)_{k}}\frac{q^{\mp\binom{k}{2}}(q^{\pm k}x)^n}{x^k}\\
        &=x^n\sum_{k=0}^{\infty}q^{(b\mp1)\binom{k}{2}}\frac{(q^{\pm n}y/x)^k}{(q;q)_{k}}\\
        &=x^n\E_{b\mp1}(q^{\pm n}y/x;q).
    \end{align*}
If the derivative $\D_{q}$ is used and if $b=1,2,3$, then
\begin{align*}
    \E_{q}(y\D_{q}|q)\{x^n\}&=\frac{x^n}{(q^ny/x;q)_{\infty}}=x^n\frac{(y/x;q)_{n}}{(y/x;q)_{\infty}}.\\
    \E_{q}(y\D_{q}|q^2)\{x^n\}&=x^n(-q^ny/x;q)_{\infty}=x^n\frac{(-y/x;q)_{\infty}}{(y/x;q)_{n}}.\\
    \E_{q}(y\D_{q}|q)\{x^n\}&=x^n\K_{\infty}(q^ny/x;q).
\end{align*}
However, if the derivative $\D_{q^{-1}}$ is used and if $b=0,1$, then
\begin{align*}
    \E_{q}(y\D_{q^{-1}}|1)\{x^n\}&=x^{n}(q^{-n}y/x;q)_{\infty}=q^{-\binom{n+1}{2}}y^n(-qx/y;q)_{n}(-y/x;q)_{\infty}.\\
    \E_{q}(y\D_{q^{-1}}|q)\{x^n\}&=x^{n}\K_{\infty}(q^{-n}y/x).
\end{align*}
\end{proof}

\begin{theorem}\label{theo_operqb_theta}
For $b\geq2$
    \begin{align}\label{eqn_theo3_1}
        \E_{q}(y\D_{q}|q^b)\{\vartheta(ax;q)\}&=\vartheta(ax;q)\E_{b-2}(y/qax^2;q)
    \end{align}
and for $b\geq0$
\begin{equation}\label{eqn_theo3_2}
    \E_{q}(y\D_{q^{-1}}|q^b)\{\vartheta(ax;q)\}=\vartheta(ax;q)\E_{b}(ay;q).
\end{equation}
\end{theorem}
\begin{proof}
By Definition \ref{def_PTO} and Proposition \ref{propo_thetaqn}, we have
    \begin{align*}
        \E_{q}(y\D_{q}|q^b)\{\vartheta(ax;q)\}&=\sum_{n=0}^{\infty}q^{b\binom{n}{2}}\frac{y^n}{(q;q)_{n}}\D_{q}^n\{\vartheta(ax;q)\}\\
        &=\sum_{n=0}^{\infty}q^{b\binom{n}{2}}\frac{y^n}{(q;q)_{n}}\frac{\vartheta(ax;q)}{(qax^2)^nq^{2\binom{n}{2}}}\\
        &=\vartheta(ax;q)\sum_{n=0}^{\infty}q^{(b-2)\binom{n}{2}}\frac{(y/qax^2)^n}{(q;q)_{n}}\\
        &=\vartheta(ax;q)\E_{b-2}(y/qax^2;q)
    \end{align*}
and    
    \begin{align*}
        \E_{q}(y\D_{q^{-1}}|q^b)\{\vartheta(ax;q)\}&=\sum_{n=0}^{\infty}q^{b\binom{n}{2}}\frac{y^n}{(q;q)_{n}}\D_{q^{-1}}^n\vartheta(ax;q)\\
        &=\sum_{n=0}^{\infty}q^{b\binom{n}{2}}\frac{y^n}{(q;q)_{n}}a^n\vartheta(ax;q)\\
        &=\vartheta(ax;q)\sum_{n=0}^{\infty}q^{b\binom{n}{2}}\frac{(ay)^n}{(q;q)_{n}}\\
        &=\vartheta(ax;q)\E_{b}(ay;q).
    \end{align*}
\end{proof}
If $b=2,3,4$ in Eq.(\ref{eqn_theo3_1}), then
    \begin{align}
        \E_{q}(y\D_{q}|q^2)\{\vartheta\}&=\frac{\vartheta(ax;q)}{(y/qax^2)_{\infty}}.\\
        \E_{q}(y\D_{q}|q^3)\{\vartheta\}&=\vartheta(ax;q)(-y/qa^2;q)_{\infty}.\\
        \E_{q}(y\D_{q}|q^4)\{\vartheta\}&=\vartheta(ax;q)\K_{\infty}(y/qa^2).
    \end{align}
If $b=0,1,2$ in Eq.(\ref{eqn_theo3_2}), then
    \begin{align}
        \E_{q}(y\D_{q^{-1}}|1)\{\vartheta(ax;q)\}&=\frac{\vartheta(ax;q)}{(ay;q)_{\infty}}.\label{eqn_oper1_theta}\\
        \E_{q}(y\D_{q^{-1}}|q)\{\vartheta(ax;q)\}&=\vartheta(ax;q)(-ay;q)_{\infty}.\label{eqn_operq_theta}\\
        \E_{q}(y\D_{q^{-1}}|q^2)\{\vartheta(ax;q)\}&=\vartheta(ax;q)\K_{\infty}(ay).\label{eqn_operq2_theta}
    \end{align}

\begin{theorem}
We have that
\begin{equation}
    \E_{q}(y\D_{q}|q^2)\{x^n(a/x;q)_{\infty}\}=x^n\frac{(-ay/qx^2;q)_{n}}{(-y/x;q)_{n}}\frac{(-y/x;q)_{\infty}}{(-ay/qx^2;q)_{\infty}}.
\end{equation}
For $b\geq3$,
\begin{align}
    \E_{q}(y\D_{q}|q^b)\{x^n(a/x;q)_{\infty}\}=x^n(a/x;q)_{\infty}\cdot{}_{1}\phi_{b-2}\left(
    \begin{array}{c}
         qx/a \\
         \mathbf{0}_{b-2}
    \end{array}
    ;q, (-1)^{b-1}q^{n-1}ay/x^2
    \right).
\end{align}
For $b\geq0$,
    \begin{multline}
        \E_{q}(y\D_{q^{-1}}|q^b)\{x^n(a/x;q)_{\infty}\}\\
        =x^n(a/x;q)_{\infty}\cdot{}_{1}\phi_{b+1}\left(
    \begin{array}{c}
         0 \\
         a/x,\mathbf{0}_{b}
    \end{array}
    ;q, (-1)^{b+1}ay/q^{n}x
    \right).
    \end{multline}
\end{theorem}
\begin{proof}
From Lemma \ref{lemma1}, we have that
\begin{align*}
    \E_{q}(y\D_{q}|q^2)\{x^n(a/x;q)_{\infty}\}&=\sum_{k=0}^{\infty}q^{2\binom{k}{2}}\frac{y^{k}}{(q;q)_{k }}\D_{q}^k\{x^n(a/x;q)_{\infty}\}\\
    &=\sum_{k=0}^{\infty}q^{2\binom{k}{2}}\frac{y^{k}}{(q;q)_{k}}x^{n-2k}\frac{(-a)^k}{q^{k^2-nk}}(qx/a;q)_{k}(a/x;q)_{\infty}\\
    &=x^n(qx/a;q)_{\infty}\sum_{k=0}^{\infty}\frac{(qx/a;q)_{k}}{(q;q)_{k}}(-q^{n-1}ay/x^2)^k\\
    &=x^n(qx/a;q)_{\infty}\frac{(-q^{n}y/x;q)_{\infty}}{(-q^{n-1}ay/x^2;q)_{\infty}}\\
    &=x^n\frac{(-ay/qx^2;q)_{n}}{(-y/x;q)_{n}}\frac{(-y/x;q)_{\infty}}{(-ay/qx^2;q)_{\infty}}.
\end{align*}
For $b\geq3$,
    \begin{align*}
        \E_{q}(y\D_{q}|q^b)\{x^n(a/x;q)_{\infty}\}
        &=\sum_{k=0}^{\infty}q^{b\binom{k}{2}}\frac{y^{k}}{(q;q)_{k }}x^{n-2k}\frac{(-a)^k}{q^{k^2-nk}}(qx/a;q)_{k}(a/x;q)_{\infty}\\
        &=x^n(a/x;q)_{\infty}\cdot{}_{1}\phi_{b-2}\left(
    \begin{array}{c}
         qx/a \\
         \mathbf{0}_{b-2}
    \end{array}
    ;q, (-1)^{b-1}q^{n-1}ay/x^2
    \right).
    \end{align*}
On the other hand, for $b\geq0$
    \begin{align*}
        \E_{q}(y\D_{q^{-1}}|q^b)\{x^n(a/x;q)_{\infty}\}
        &=\sum_{k=0}^{\infty}q^{b\binom{k}{2}}\frac{y^{k}}{(q;q)_{k }}\D_{q^{-1}}^k\{x^n(a/x;q)_{\infty}\}\\
        &=\sum_{k=0}^{\infty}q^{b\binom{k}{2}}\frac{y^k}{(q;q)_{k}}x^n\frac{q^{\binom{k}{2}-nk}}{x^k}\frac{(a/x;q)_{\infty}}{(a/x;q)_{k}}\\
        &=x^n(a/x;q)_{\infty}\cdot{}_{1}\phi_{b+1}\left(
    \begin{array}{c}
         0 \\
         a/x,0,\ldots,0
    \end{array}
    ;q, (-1)^{b+1}y/q^{n}x
    \right).
    \end{align*}
The proof is completed.
\end{proof}

\begin{theorem}
We have that
\begin{equation}
    \E_{q}(y\D_{q}|q^3)\left\{(a/x;q)_{\infty}\vartheta(cx;q)\right\}=\vartheta(cx;q)\frac{(a/x;q)_{\infty}(-y/qcx^2;q)_{\infty}}{(-ay/q^2cx^3;q)_{\infty}}.
\end{equation}
For $b\geq4$,
\begin{multline}
    \E_{q}(y\D_{q}|q^b)\left\{(a/x;q)_{\infty}\vartheta(cx;q)\right\}\\
    =(a/x;q)_{\infty}\vartheta(cx;q){}_{1}\phi_{b-3}\left(
    \begin{array}{c}
         qx/a \\
         \mathbf{0}_{b-3}
    \end{array}
    ;q, (-1)^{b-2}ay/cx^3
    \right).
\end{multline}
We have that
\begin{equation}
    \E_{q}(y\D_{q^{-1}}|1)\left\{(a/x;q)_{\infty}\vartheta(cx;q)\right\}=(a/x;q)_{\infty}\vartheta(cx;q){}_{2}\phi_{1}\left(
    \begin{array}{c}
         0,0 \\
         a/x
    \end{array}
    ;q, cy
    \right).
\end{equation}
For $b\geq1$,
    \begin{multline}
        \E_{q}(y\D_{q^{-1}}|q^b)\left\{(a/x;q)_{\infty}\vartheta(cx;q)\right\}\\
        =(a/x;q)_{\infty}\vartheta(cx;q){}_{1}\phi_{b}\left(
    \begin{array}{c}
         0 \\
         a/x,\mathbf{0}_{b-1}
    \end{array}
    ;q, (-1)^{b}cy
    \right).
    \end{multline}
\end{theorem}
\begin{proof}
From Leibniz $\lambda$-rule Eq.(\ref{prop_leibniz})
    \begin{align*}
        &\E_{q}(y\D_{q^{\pm}}|q^b)\left\{(a/x;q)_{\infty}\vartheta(cx;q)\right\}\\
        &\hspace{1cm}=\sum_{n=0}^{\infty}q^{b\binom{n}{2}}\frac{y^{n}}{(q;q)_{n}}\D_{q^{\pm}}^n\{(a/x;q)_{\infty}\vartheta(cx;q)\}\\
        &\hspace{1cm}=\sum_{n=0}^{\infty}q^{b\binom{n}{2}}\frac{y^{n}}{(q;q)_{n}}q^{\pm\binom{n}{2}}x^n\D_{q^{\pm}}^n\{(a/x;q)_{\infty}\}\D_{q^{\pm}}^n\{\vartheta(cx;q)\}.
    \end{align*}
Now, by using Lemma \ref{lemma1} and Proposition \ref{propo_thetaqn}, and with $b=3$ we have that
\begin{align*}
    &\E_{q}(y\D_{q}|q^3)\left\{(a/x;q)_{\infty}\vartheta(cx;q)\right\}\\
    &\hspace{1cm}=(a/x;q)_{\infty}\vartheta(cx;q)\sum_{n=0}^{\infty}\frac{(qx/a;q)_{n}}{(q;q)_{n}}(-ay/q^2cx^3)^{n}\\
    &\hspace{1cm}=\vartheta(cx;q)\frac{(a/x;q)_{\infty}(-y/qcx^2;q)_{\infty}}{(-ay/q^2cx^3;q)_{\infty}}.
\end{align*}
If $b\geq4$, then
\begin{align*}
        &\E_{q}(y\D_{q}|q^b)\left\{(a/x;q)_{\infty}\vartheta(cx;q)\right\}\\
        &\hspace{1cm}=(a/x;q)_{\infty}\vartheta(cx;q)\sum_{n=0}^{\infty}q^{(b-3)\binom{n}{2}}\frac{(-ay/cx^3)^{n}}{(q;q)_{n}}(qx/a;q)_{n}\\
        &\hspace{1cm}=(a/x;q)_{\infty}\vartheta(cx;q){}_{1}\phi_{b-3}\left(
    \begin{array}{c}
         qx/a \\
         \mathbf{0}_{b-3}
    \end{array}
    ;q, (-1)^{b-2}ay/cx^3
    \right).
    \end{align*}
On the other hand,
\begin{align*}
    &\E_{q}(y\D_{q^{-1}}|1)\left\{(a/x;q)_{\infty}\vartheta(cx;q)\right\}\\
    &\hspace{1cm}=(a/x;q)_{\infty}\vartheta(cx;q)\sum_{n=0}^{\infty}\frac{(cy)^n}{(q;q)_{n}(a/x;q)_{n}}\\
    &\hspace{1cm}=(a/x;q)_{\infty}\vartheta(cx;q){}_{2}\phi_{1}\left(
    \begin{array}{c}
         0,0 \\
         a/x
    \end{array}
    ;q, cy
    \right).
\end{align*}
For $b\geq1$
\begin{align*}
    &\E_{q}(y\D_{q^{-1}}|q^b)\left\{(a/x;q)_{\infty}\vartheta(cx;q)\right\}\\
    &\hspace{1cm}=(a/x;q)_{\infty}\vartheta(cx;q)\sum_{n=0}^{\infty}q^{b\binom{n}{2}}\frac{(cy)^n}{(q;q)_{n}(a/x;q)_{n}}\\
    &\hspace{1cm}=(a/x;q)_{\infty}\vartheta(cx;q){}_{1}\phi_{b}\left(
    \begin{array}{c}
         0 \\
         a/x,\mathbf{0}_{b-1}
    \end{array}
    ;q, (-1)^{b}cy
    \right).
\end{align*}
\end{proof}

\begin{theorem}\label{theo6}
We have that
\begin{equation}
    \E_{q}(y\D_{q}|q)\bigg\{x^n\frac{(ax,q/ax;q)_{\infty}}{(x,b/ax;q)_{\infty}}\bigg\}
    =x^n\frac{(ax,q/ax;q)_{\infty}}{(x,b/ax;q)_{\infty}}\cdot{}_{2}\phi_{1}\left(
    \begin{array}{c}
         x,0 \\
         qax/b
    \end{array}
    ;q,q^{n+1}y
    \right).
\end{equation}
For $c\geq1$,
    \begin{multline}
        \E_{q}(y\D_{q}|q^c)\bigg\{x^n\frac{(ax,q/ax;q)_{\infty}}{(x,b/ax;q)_{\infty}}\bigg\}\\=x^n\frac{(ax,q/ax;q)_{\infty}}{(x,b/ax;q)_{\infty}}\cdot{}_{1}\phi_{c-1}\left(
    \begin{array}{c}
         x \\
         qax/b,\mathbf{0}_{c-2}
    \end{array}
    ;q,(-1)^{c-1}q^{n+1}y
    \right).
    \end{multline}
\end{theorem}
\begin{proof}
From Lemma \ref{lemma2}
    \begin{align*}
        &\E_{q}(y\D_{q}|q^c)\bigg\{x^n\frac{(ax,q/ax;q)_{\infty}}{(x,b/ax;q)_{\infty}}\bigg\}\\
        &=\sum_{k=0}^{\infty}q^{c\binom{k}{2}}\frac{y^{k}}{(q;q)_{k}}\D_{q}^k\left\{x^n\frac{(ax,q/ax;q)_{\infty}}{(x,b/ax;q)_{\infty}}\right\}\\
        &=\sum_{k=0}^{\infty}q^{c\binom{k}{2}}\frac{y^k}{(q;q)_{k}}(q^kx)^n\frac{q^k(x;q)_{k}}{q^{\binom{k}{2}}(qax/b;q)_{k}}\frac{(ax,q/ax;q)_{\infty}}{(x,b/ax;q)_{\infty}}\\
        &=x^n\frac{(ax,q/ax;q)_{\infty}}{(x,b/ax;q)_{\infty}}\sum_{k=0}^{\infty}q^{(c-1)\binom{k}{2}}\frac{(x;q)_{k}}{(qax/b;q)_{k}(q;q)_{k}}(q^{n+1}y)^k.
    \end{align*}
If $c=1$, then
\begin{align*}
    &\E_{q}(y\D_{q}|q)\bigg\{x^n\frac{(ax,q/ax;q)_{\infty}}{(x,b/ax;q)_{\infty}}\bigg\}\\
    &\hspace{1cm}=x^n\frac{(ax,q/ax;q)_{\infty}}{(x,b/ax;q)_{\infty}}\cdot{}_{2}\phi_{1}\left(
    \begin{array}{c}
         x,0 \\
         qax/b
    \end{array}
    ;q,q^{n+1}y
    \right).
\end{align*}
If $c\geq2$, then
\begin{align*}
    &\E_{q}(y\D_{q}|q^c)\bigg\{x^n\frac{(ax,q/ax;q)_{\infty}}{(x,b/ax;q)_{\infty}}\bigg\}\\
    &\hspace{1cm}=x^n\frac{(ax,q/ax;q)_{\infty}}{(x,b/ax;q)_{\infty}}\cdot{}_{1}\phi_{c-1}\left(
    \begin{array}{c}
         x \\
         qax/b,\mathbf{0}_{c-2}
    \end{array}
    ;q,(-1)^{c-1}q^{n+1}y
    \right).
\end{align*}
The proof is completed.
\end{proof}

\begin{theorem}
We have that
\begin{multline}
    \E_{q}(y\D_{q}|q^2)\bigg\{\vartheta(dx;q)\frac{(ax,q/ax;q)_{\infty}}{(x,b/ax;q)_{\infty}}\bigg\}\\
    =\vartheta(dx;q)\frac{(ax,q/ax;q)_{\infty}}{(x,b/ax;q)_{\infty}}\cdot{}_{2}\phi_{1}\left(
    \begin{array}{c}
         x,0 \\
         qax/b
    \end{array}
    ;q,y/dx
    \right).
\end{multline}
For $c\geq3$,
    \begin{multline}
        \E_{q}(y\D_{q}|q^c)\bigg\{\vartheta(dx;q)\frac{(ax,q/ax;q)_{\infty}}{(x,b/ax;q)_{\infty}}\bigg\}\\=\vartheta(dx;q)\frac{(ax,q/ax;q)_{\infty}}{(x,b/ax;q)_{\infty}}\cdot{}_{1}\phi_{c-2}\left(
    \begin{array}{c}
         x \\
         qax/b,\mathbf{0}_{c-3}
    \end{array}
    ;q,(-1)^{c-2}y/dx
    \right).
    \end{multline}
\end{theorem}
\begin{proof}
From 
    \begin{align*}
        &\E_{q}(y\D_{q}|q^c)\bigg\{\vartheta(dx;q)\frac{(ax,q/ax;q)_{\infty}}{(x,b/ax;q)_{\infty}}\bigg\}\\
        &=\sum_{k=0}^{\infty}q^{c\binom{k}{2}}\frac{y^{k}}{(q;q)_{k}}\D_{q}^k\left\{\vartheta(dx;q)\frac{(ax,q/ax;q)_{\infty}}{(x,b/ax;q)_{\infty}}\right\}\\
        &=\sum_{k=0}^{\infty}q^{c\binom{k}{2}}\frac{y^k}{(q;q)_{k}}q^{\binom{k}{2}}x^k\frac{\vartheta(dx;q)}{(qdx^2)^kq^{2\binom{k}{2}}}\frac{q^k(x;q)_{k}}{q^{\binom{k}{2}}(qax/b;q)_{k}}\frac{(ax,q/ax;q)_{\infty}}{(x,b/ax;q)_{\infty}}\\
        &=\vartheta(dx;q)\frac{(ax,q/ax;q)_{\infty}}{(x,b/ax;q)_{\infty}}\sum_{k=0}^{\infty}q^{(c-2)\binom{k}{2}}\frac{(x;q)_{k}}{(qax/b;q)_{k}(q;q)_{k}}(y/dx)^k.
    \end{align*}
If $c=2$, then
\begin{align*}
    &\E_{q}(y\D_{q}|q^2)\bigg\{\vartheta(dx;q)\frac{(ax,q/ax;q)_{\infty}}{(x,b/ax;q)_{\infty}}\bigg\}\\
    &=\vartheta(dx;q)\frac{(ax,q/ax;q)_{\infty}}{(x,b/ax;q)_{\infty}}\cdot{}_{2}\phi_{1}\left(
    \begin{array}{c}
         x,0 \\
         qax/b
    \end{array}
    ;q,y/dx
    \right).
\end{align*}
If $c\geq3$, then
\begin{align*}
    &\E_{q}(y\D_{q}|q^c)\bigg\{\vartheta(dx;q)\frac{(ax,q/ax;q)_{\infty}}{(x,b/ax;q)_{\infty}}\bigg\}\\
    &=\vartheta(dx;q)\frac{(ax,q/ax;q)_{\infty}}{(x,b/ax;q)_{\infty}}\cdot{}_{1}\phi_{c-2}\left(
    \begin{array}{c}
         x \\
         qax/b,\mathbf{0}_{c-3}
    \end{array}
    ;q,(-1)^{c-2}y/dx
    \right).
\end{align*}
The proof is completed.
\end{proof}

\begin{theorem}
For $c\geq0$,
    \begin{multline}
        \E_{q}(y\D_{q^{-1}}|q^c)\bigg\{x^n\frac{(ax,q/ax;q)_{\infty}}{(x,b/ax;q)_{\infty}}\bigg\}\\=x^n\frac{(ax,q/ax;q)_{\infty}}{(x,b/ax;q)_{\infty}}\cdot{}_{1}\phi_{c+1}\left(
    \begin{array}{c}
         b/ax \\
         q/x,\mathbf{0}_{c}
    \end{array}
    ;q,(-1)^{c+1}aq^{n}y/x
    \right).
    \end{multline}
\end{theorem}
\begin{proof}
From Lemma \ref{lemma2}
    \begin{align*}
        &\E_{q}(y\D_{q^{-1}}|q^c)\bigg\{x^n\frac{(ax,q/ax;q)_{\infty}}{(x,b/ax;q)_{\infty}}\bigg\}\\
        &=\sum_{k=0}^{\infty}q^{c\binom{k}{2}}\frac{y^{k}}{(q;q)_{k}}\D_{q}^k\left\{x^n\frac{(ax,q/ax;q)_{\infty}}{(x,b/ax;q)_{\infty}}\right\}\\
        &=\sum_{k=0}^{\infty}q^{c\binom{k}{2}}\frac{y^k}{(q;q)_{k}}(q^kx)^n\frac{q^{\binom{k}{2}}}{x^k}\frac{a^k(b/ax;q)_{k}}{(q/x;q)_{k}}\frac{(ax,q/ax;q)_{\infty}}{(x,b/ax;q)_{\infty}}\\
        &=x^n\frac{(ax,q/ax;q)_{\infty}}{(x,b/ax;q)_{\infty}}\sum_{k=0}^{\infty}q^{(c+1)\binom{k}{2}}\frac{(b/ax;q)_{k}}{(q/x;q)_{k}(q;q)_{k}}(aq^{n}y/x)^k\\
        &=x^n\frac{(ax,q/ax;q)_{\infty}}{(x,b/ax;q)_{\infty}}\cdot{}_{1}\phi_{c+1}\left(
    \begin{array}{c}
         b/ax \\
         q/x,\mathbf{0}_{c}
    \end{array}
    ;q,(-1)^{c+1}aq^{n}y/x
    \right).
    \end{align*}
The proof is completed.
\end{proof}

\begin{theorem}
We have that
\begin{multline}
    \E_{q}(y\D_{q^{-1}}|1)\bigg\{\vartheta(dx;q)\frac{(ax,q/ax;q)_{\infty}}{(x,b/ax;q)_{\infty}}\bigg\}\\
    =\vartheta(dx;q)\frac{(ax,q/ax;q)_{\infty}}{(x,b/ax;q)_{\infty}}\cdot{}_{2}\phi_{1}\left(
    \begin{array}{c}
         b/ax,0 \\
         q/x
    \end{array}
    ;q,ady
    \right).
\end{multline}
For $c\geq1$,
    \begin{multline}
        \E_{q}(y\D_{q^{-1}}|q^c)\bigg\{\vartheta(dx;q)\frac{(ax,q/ax;q)_{\infty}}{(x,b/ax;q)_{\infty}}\bigg\}\\=\vartheta(dx;q)\frac{(ax,q/ax;q)_{\infty}}{(x,b/ax;q)_{\infty}}\cdot{}_{1}\phi_{c}\left(
    \begin{array}{c}
         b/ax \\
         q/x,\mathbf{0}_{c-1}
    \end{array}
    ;q,(-1)^{c}ady
    \right).
    \end{multline}    
\end{theorem}
\begin{proof}
    \begin{align*}
        &\E_{q}(y\D_{q^{-1}}|q^c)\bigg\{\vartheta(dx;q)\frac{(ax,q/ax;q)_{\infty}}{(x,b/ax;q)_{\infty}}\bigg\}\\
        &=\sum_{n=0}^{\infty}q^{c\binom{n}{2}}\frac{y^n}{(q;q)_{n}}\D_{q^{-1}}^n\left\{\vartheta(dx;q)\frac{(ax,q/ax;q)_{\infty}}{(x,b/ax;q)_{\infty}}\right\}\\
        &=\sum_{n=0}^{\infty}q^{c\binom{n}{2}}\frac{y^n}{(q;q)_{n}}q^{-\binom{n}{2}}x^n\D_{q^{-1}}^n\{\vartheta(dx;q)\}\D_{q^{-1}}^n\left\{\frac{(ax,q/ax;q)_{\infty}}{(x,b/ax;q)_{\infty}}\right\}\\
        &=\sum_{n=0}^{\infty}q^{c\binom{n}{2}}\frac{y^n}{(q;q)_{n}}q^{-\binom{n}{2}}x^nd^n\vartheta(dx;q)\frac{q^{\binom{n}{2}}}{x^n}\frac{a^n(b/ax;q)_{n}}{(q/x;q)_{n}}\frac{(ax,q/ax;q)_{\infty}}{(x,b/ax;q)_{\infty}}\\
        &=\vartheta(dx;q)\frac{(ax,q/ax;q)_{\infty}}{(x,b/ax;q)_{\infty}}\sum_{n=0}^{\infty}q^{c\binom{n}{2}}\frac{(b/ax;q)_{n}}{(q/x;q)_{n}(q;q)_{n}}(ady)^n.
    \end{align*}
If $c=0$, then
\begin{align*}
&\E_{q}(y\D_{q^{-1}}|1)\bigg\{\vartheta(dx;q)\frac{(ax,q/ax;q)_{\infty}}{(x,b/ax;q)_{\infty}}\bigg\}\\
    &=\vartheta(dx;q)\frac{(ax,q/ax;q)_{\infty}}{(x,b/ax;q)_{\infty}}\cdot{}_{2}\phi_{1}\left(
    \begin{array}{c}
         b/ax,0 \\
         q/x
    \end{array}
    ;q,ady
    \right).
\end{align*}
If $c\geq1$, then
\begin{align*}
&\E_{q}(y\D_{q^{-1}}|q^c)\bigg\{\vartheta(dx;q)\frac{(ax,q/ax;q)_{\infty}}{(x,b/ax;q)_{\infty}}\bigg\}\\
    &=\vartheta(dx;q)\frac{(ax,q/ax;q)_{\infty}}{(x,b/ax;q)_{\infty}}\cdot{}_{1}\phi_{c}\left(
    \begin{array}{c}
         b/ax \\
         q/x,\mathbf{0}_{c-1}
    \end{array}
    ;q,(-1)^{c}ady
    \right).
\end{align*}
\end{proof}

\begin{theorem}
We have for $c\geq1$
\begin{multline}
        \E_{q}(y\D_{q}|q^c)\left\{x^n\bigg[
    \begin{array}{c}
         a_{1},\ldots,a_{r+c-1}  \\
         b_{1},\ldots,b_{r+1} 
    \end{array}
    ;q,x
    \bigg]_{\infty}\right\}\\
        =x^n\bigg[
    \begin{array}{c}
         a_{1},\ldots,a_{r+c-1}  \\
         b_{1},\ldots,b_{r+1} 
    \end{array}
    ;q,x
    \bigg]_{\infty}{}_{r+1}\phi_{r+c-1}\left(
    \begin{array}{c}
         b_{1}x,\ldots,b_{r+1}x, \\
         a_{1}x,\ldots,a_{r+c-1}x
    \end{array}
    ;q,(-1)^{c-1}q^ny/x
    \right).
    \end{multline}    
\end{theorem}
\begin{proof}
    \begin{align*}
        &\E_{q}(y\D_{q}|q^c)\left\{x^n\bigg[
    \begin{array}{c}
         a_{1},\ldots,a_{r+c-1}  \\
         b_{1},\ldots,b_{r+1} 
    \end{array}
    ;q,x
    \bigg]_{\infty}\right\}\\
        &\hspace{1cm}=\sum_{k=0}^{\infty}q^{c\binom{k}{2}}\frac{y^k}{(q;q)_{k}}\D_{q}^{k}\left\{x^n\bigg[
    \begin{array}{c}
         a_{1},\ldots,a_{r+c-1}  \\
         b_{1},\ldots,b_{r+1} 
    \end{array}
    ;q,x
    \bigg]_{\infty}\right\}\\
    &\hspace{1cm}=\sum_{k=0}^{\infty}q^{c\binom{k}{2}}\frac{y^k}{(q;q)_{k}}\frac{x^{n-k}}{q^{\binom{k}{2}-kn}}\bigg[
    \begin{array}{c}
         b_{1},\ldots,b_{r+1}  \\
         a_{1},\ldots,a_{r+c-1} 
    \end{array}
    ;q,x
    \bigg]_{k}\bigg[
    \begin{array}{c}
         a_{1},\ldots,a_{r+c-1}  \\
         b_{1},\ldots,b_{r+1} 
    \end{array}
    ;q,x
    \bigg]_{\infty}\\
    &\hspace{1cm}=x^n\bigg[
    \begin{array}{c}
         a_{1},\ldots,a_{r+c-1}  \\
         b_{1},\ldots,b_{r+1} 
    \end{array}
    ;q,x
    \bigg]_{\infty}\sum_{k=0}^{\infty}q^{(c-1)\binom{k}{2}}\frac{1}{(q;q)_{k}}\bigg[
    \begin{array}{c}
         b_{1},\ldots,b_{r+1}  \\
         a_{1},\ldots,a_{r+c-1} 
    \end{array}
    ;q,x
    \bigg]_{k}(q^ny/x)^k\\
    &\hspace{1cm}=x^n\bigg[
    \begin{array}{c}
         a_{1},\ldots,a_{r+c-1}  \\
         b_{1},\ldots,b_{r+1} 
    \end{array}
    ;q,x
    \bigg]_{\infty}{}_{r+1}\phi_{r+c-1}\left(
    \begin{array}{c}
         b_{1}x,\ldots,b_{r+1}x, \\
         a_{1}x,\ldots,a_{r+c-1}x
    \end{array}
    ;q,(-1)^{c-1}q^ny/x
    \right).
    \end{align*}
\end{proof}

\begin{corollary}
    \begin{equation}
        \E_{q}(y\D_{q}|q)\{x^n(ax;q)_{\infty}\}=x^n(ax;q)_{\infty}\cdot{}_{2}\phi_{1}\left(
    \begin{array}{c}
         0,0 \\
         ax
    \end{array}
    ;q,q^ny/x
    \right).
    \end{equation}
\end{corollary}

\begin{corollary}
    \begin{equation}
        \E_{q}(y\D_{q}|q)\left\{\frac{x^n}{(ax;q)_{\infty}}\right\}=x^n\frac{(y/x;q)_{n}}{(ay/x;q)_{n}}\frac{(ay/x;q)_{\infty}}{(ax,y/x;q)_{\infty}}.
    \end{equation}
For $c\geq2$,
    \begin{equation}
        \E_{q}(y\D_{q}|q^c)\left\{\frac{x^n}{(ax;q)_{\infty}}\right\}=\frac{x^n}{(ax;q)_{\infty}}{}_{1}\phi_{c-1}\left(
    \begin{array}{c}
         ax \\
         \mathbf{0}_{c-1}
    \end{array}
    ;q,(-1)^{c-1}q^ny/x
    \right).
    \end{equation}
\end{corollary}

\begin{corollary}
    \begin{equation}
        \E_{q}(y\D_{q}|q^c)\left\{x^n\frac{(ax;q)_{\infty}}{(bx;q)_{\infty}}\right\}=x^n\frac{(ax;q)_{\infty}}{(bx;q)_{\infty}}\cdot{}_{2}\phi_{c}\left(
    \begin{array}{c}
         bx,0 \\
         ax,\mathbf{0}_{c-1}
    \end{array}
    ;q,(-1)^{c-1}q^ny/x
    \right).
    \end{equation}
\end{corollary}

\begin{theorem}
For $c\geq2$ 
    \begin{multline}
        \E_{q}(y\D_{q}|q^c)\left\{\vartheta(dx;q)\bigg[
    \begin{array}{c}
         a_{1},\ldots,a_{r+c-2}  \\
         b_{1},\ldots,b_{r+1} 
    \end{array}
    ;q,x
    \bigg]_{\infty}\right\}\\
        =\vartheta(dx;q)\bigg[
    \begin{array}{c}
         a_{1},\ldots,a_{r+c-2}  \\
         b_{1},\ldots,b_{r+1} 
    \end{array}
    ;q,x
    \bigg]_{\infty}{}_{r+1}\phi_{r+c-2}\left(
    \begin{array}{c}
         b_{1}x,\ldots,b_{r+1}x, \\
         a_{1}x,\ldots,a_{r+c-2}x
    \end{array}
    ;q,(-1)^{c-2}y/qdx^2
    \right).
    \end{multline}
\end{theorem}
\begin{proof}
    \begin{align*}
        &\E_{q}(y\D_{q}|q^c)\left\{\vartheta(dx;q)\bigg[
    \begin{array}{c}
         a_{1},\ldots,a_{r+c-2}  \\
         b_{1},\ldots,b_{r+1} 
    \end{array}
    ;q,x
    \bigg]_{\infty}\right\}\\
    &=\sum_{n=0}^{\infty}q^{c\binom{n}{2}}\frac{y^n}{(q;q)_{n}}\D_{q}^n\left\{\vartheta(dx;q)\bigg[
    \begin{array}{c}
         a_{1},\ldots,a_{r+c-2}  \\
         b_{1},\ldots,b_{r+1} 
    \end{array}
    ;q,x
    \bigg]_{\infty}\right\}\\
    &=\sum_{n=0}^{\infty}q^{c\binom{n}{2}}\frac{y^n}{(q;q)_{n}}q^{\binom{n}{2}}x^n\D_{q}^n\{\vartheta(dx;q)\}\D_{q}^n\left\{\bigg[
    \begin{array}{c}
         a_{1},\ldots,a_{r+c-2}  \\
         b_{1},\ldots,b_{r+1} 
    \end{array}
    ;q,x
    \bigg]_{\infty}\right\}\\
    &=\sum_{n=0}^{\infty}q^{c\binom{n}{2}}\frac{y^n}{(q;q)_{n}}\frac{q^{\binom{n}{2}}x^n\vartheta(dx;q)}{(qdx^2)^nq^{2\binom{n}{2}}}\frac{x^{-n}}{q^{\binom{n}{2}}}\bigg[
    \begin{array}{c}
         b_{1},\ldots,b_{r+1}  \\
         a_{1},\ldots,a_{r+c-2} 
    \end{array}
    ;q,x
    \bigg]_{n}\bigg[
    \begin{array}{c}
         a_{1},\ldots,a_{r+c-2}  \\
         b_{1},\ldots,b_{r+1} 
    \end{array}
    ;q,x
    \bigg]_{\infty}\\
    &=\vartheta(dx;q)\bigg[
    \begin{array}{c}
         a_{1},\ldots,a_{r+c-2}  \\
         b_{1},\ldots,b_{r+1} 
    \end{array}
    ;q,x
    \bigg]_{\infty}\sum_{n=0}^{\infty}\frac{q^{(c-2)\binom{n}{2}}}{(q;q)_{n}}\bigg[
    \begin{array}{c}
         b_{1},\ldots,b_{r+1}  \\
         a_{1},\ldots,a_{r+c-2} 
    \end{array}
    ;q,x
    \bigg]_{n}(y/qdx^2)^n\\
    &=\vartheta(dx;q)\bigg[
    \begin{array}{c}
         a_{1},\ldots,a_{r+c-2}  \\
         b_{1},\ldots,b_{r+1} 
    \end{array}
    ;q,x
    \bigg]_{\infty}{}_{r+1}\phi_{r+c-2}\left(
    \begin{array}{c}
         b_{1}x,\ldots,b_{r+1}x, \\
         a_{1}x,\ldots,a_{r+c-2}x
    \end{array}
    ;q,(-1)^{c-2}y/qdx^2
    \right).
    \end{align*}
\end{proof}

\begin{corollary}
    \begin{equation}
        \E_{q}(y\D_{q}|q^2)\{\vartheta(bx;q)(ax;q)_{\infty}\}=\vartheta(bx;q)(ax;q)_{\infty}\cdot{}_{2}\phi_{1}\left(
    \begin{array}{c}
         0,0 \\
         ax
    \end{array}
    ;q,y/bdx^2
    \right).
    \end{equation}
\end{corollary}

\begin{corollary}
    \begin{equation}
        \E_{q}(y\D_{q}|q^2)\left\{\frac{\vartheta(bx;q)}{(ax;q)_{\infty}}\right\}=\vartheta(bx;q)\frac{(ay/qbx;q)_{\infty}}{(ax,y/qbx^2;q)_{\infty}}.
    \end{equation}
For $c\geq3$,
    \begin{equation}
        \E_{q}(y\D_{q}|q^c)\left\{\frac{\vartheta(bx;q)}{(ax;q)_{\infty}}\right\}=\frac{\vartheta(bx;q)}{(ax;q)_{\infty}}{}_{1}\phi_{c-2}\left(
    \begin{array}{c}
         ax \\
         \mathbf{0}_{c-2}
    \end{array}
    ;q,(-1)^{c-2}y/qbx^2
    \right).
    \end{equation}
\end{corollary}

\begin{corollary}
    \begin{multline}
        \E_{q}(y\D_{q}|q^c)\left\{\vartheta(dx;q)\frac{(ax;q)_{\infty}}{(bx;q)_{\infty}}\right\}\\
        =\vartheta(dx;q)\frac{(ax;q)_{\infty}}{(bx;q)_{\infty}}\cdot{}_{2}\phi_{c-1}\left(
    \begin{array}{c}
         bx,0 \\
         ax,\mathbf{0}_{c-2}
    \end{array}
    ;q,(-1)^{c-2}y/qdx^2
    \right).
    \end{multline}
\end{corollary}

\begin{theorem}
For all $c\in\N$
    \begin{multline}
        \E_{q}(y\D_{q^{-1}}|q^c)\left\{x^n\bigg[
    \begin{array}{c}
         a_{1},\ldots,a_{r}  \\
         b_{1},\ldots,b_{s} 
    \end{array}
    ;q,x
    \bigg]_{\infty}\right\}\\
        =x^n\bigg[
    \begin{array}{c}
         a_{1},\ldots,a_{r}  \\
         b_{1},\ldots,b_{s} 
    \end{array}
    ;q,x
    \bigg]_{\infty}{}_{r}\phi_{s+c}\left(
    \begin{array}{c}
         q/a_{1}x,\ldots,q/a_{r}x \\
         q/b_{1}x,\ldots,q/b_{s}x,\mathbf{0}_{c}
    \end{array}
    ;q,(-1)^{c+1}\frac{q^{s-r-n}ya_{1}\cdots a_{r}}{x^{1+s-r}b_{1}\cdots b_{s}}
    \right),
    \end{multline}
\end{theorem}
\begin{proof}
    \begin{align*}
        &\E_{q}(y\D_{q^{-1}}|q^c)\left\{x^n\bigg[
    \begin{array}{c}
         a_{1},\ldots,a_{r}  \\
         b_{1},\ldots,b_{s} 
    \end{array}
    ;q,x
    \bigg]_{\infty}\right\}\\
        &=\sum_{k=0}^{\infty}q^{c\binom{k}{2}}\frac{y^{k}}{(q;q)_{k}}\D_{q}^{k}\left\{x^n\bigg[
    \begin{array}{c}
         a_{1},\ldots,a_{r}  \\
         b_{1},\ldots,b_{s} 
    \end{array}
    ;q,x
    \bigg]_{\infty}\right\}\\
        &=x^n\bigg[
    \begin{array}{c}
         a_{1},\ldots,a_{r}  \\
         b_{1},\ldots,b_{s} 
    \end{array}
    ;q,x
    \bigg]_{\infty}\sum_{k=0}^{\infty}\frac{y^{k}}{(q;q)_{k}}\big[(-1)^kq^{\binom{k}{2}}\big]^{(1+s-r+c)}\left((-1)^{c+1}\frac{q^{s-r-n}a_{1}\cdots a_{r}}{x^{1+s-r}b_{1}\cdots b_{s}}\right)^k\\
        &\hspace{4cm}\times\bigg[
    \begin{array}{c}
         q/a_{1},\ldots,q/a_{r}  \\
         q/b_{1},\ldots,q/b_{s} 
    \end{array}
    ;q,x^{-1}
    \bigg]_{k}\\
        &=x^n\bigg[
    \begin{array}{c}
         a_{1},\ldots,a_{r}  \\
         b_{1},\ldots,b_{s} 
    \end{array}
    ;q,x
    \bigg]_{\infty}{}_{r}\phi_{s+c}\left(
    \begin{array}{c}
         q/a_{1}x,\ldots,q/a_{r}x \\
         q/b_{1}x,\ldots,q/b_{s}x,\mathbf{0}_{c}
    \end{array}
    ;q,(-1)^{c+1}\frac{q^{s-r-n}ya_{1}\cdots a_{r}}{x^{1+s-r}b_{1}\cdots b_{s}}
    \right).
    \end{align*}
\end{proof}

\begin{corollary}\label{coro10}
We have that
\begin{equation}
    \E_{q}(y\D_{q^{-1}}|1)\{x^n(ax;q)_{\infty}\}=(qx/a)^n\frac{(-qx/y;q)_{n}}{(-q^2/ay;q)_{n}}\frac{(ax,-y/x;q)_{\infty}}{(-ay/q;q)_{\infty}}.
\end{equation}
For all $c\geq1$
    \begin{equation}
        \E_{q}(y\D_{q^{-1}}|q^c)\{x^n(ax;q)_{\infty}\}
        =x^n(ax;q)_{\infty}\cdot{}_{1}\phi_{c}\left(
    \begin{array}{c}
         q/ax\\
         \mathbf{0}_{c}
    \end{array}
    ;q,(-1)^{c+1}ay/q^{n+1}
    \right).
    \end{equation}
\end{corollary}

\begin{corollary}\label{coro11}
For $c\geq0$
    \begin{equation}
        \E_{q}(y\D_{q^{-1}}|q^c)\left\{\frac{x^n}{(ax;q)_{\infty}}\right\}=\frac{x^n}{(ax;q)_{\infty}}{}_{0}\phi_{c+1}\left(
    \begin{array}{c}
         -\\
         q/ax,\mathbf{0}_{c}
    \end{array}
    ;q,(-1)^{c+1}y/aq^{n}x^2
    \right).
    \end{equation}
\end{corollary}

\begin{corollary}
    \begin{equation}
        \E_{q}(y\D_{q^{-1}}|q^c)\left\{x^n\frac{(ax;q)_{\infty}}{(bx;q)_{\infty}}\right\}=x^n\frac{(ax;q)_{\infty}}{(bx;q)_{\infty}}{}_{1}\phi_{c+1}\left(
    \begin{array}{c}
         q/ax\\
         q/bx,\mathbf{0}_{c}
    \end{array}
    ;q,(-1)^{c}axy/bq^{n}
    \right).
    \end{equation}
\end{corollary}

\begin{theorem}
For $c\in\N$
    \begin{multline}
        \E_{q}(y\D_{q^{-1}}|q^c)\left\{\vartheta(dx;q)\bigg[
    \begin{array}{c}
         a_{1},\ldots,a_{r}  \\
         b_{1},\ldots,b_{s} 
    \end{array}
    ;q,x
    \bigg]_{\infty}\right\}\\
    =\vartheta(dx;q)\bigg[
    \begin{array}{c}
         a_{1},\ldots,a_{r}  \\
         b_{1},\ldots,b_{s} 
    \end{array}
    ;q,x
    \bigg]_{\infty}{}_{r+1}\phi_{s+c}\left(
    \begin{array}{c}
         q/a_{1}x,\ldots,q/a_{r}x,0, \\
         q/b_{1}x,\ldots,q/b_{s}x,\mathbf{0}_{c}
    \end{array}
    ;q,(-1)^cy\frac{q^{s-r}a_{1}\cdots a_{r}}{x^{s-r}b_{1}\cdots b_{s}}
    \right).
    \end{multline}
\end{theorem}
\begin{proof}
    \begin{align*}
        &\E_{q}(y\D_{q^{-1}}|q^c)\left\{\vartheta(dx;q)\bigg[
    \begin{array}{c}
         a_{1},\ldots,a_{r}  \\
         b_{1},\ldots,b_{s} 
    \end{array}
    ;q,x
    \bigg]_{\infty}\right\}\\
    &=\sum_{n=0}^{\infty}q^{c\binom{n}{2}}\frac{y^n}{(q;q)_{n}}\D_{q^{-1}}^n\left\{\vartheta(cx;q)\bigg[
    \begin{array}{c}
         a_{1},\ldots,a_{r}  \\
         b_{1},\ldots,b_{s} 
    \end{array}
    ;q,x
    \bigg]_{\infty}\right\}\\
    &=\sum_{n=0}^{\infty}q^{c\binom{n}{2}}\frac{y^n}{(q;q)_{n}}q^{-\binom{n}{2}}x^n\D_{q^{-1}}^n\{\vartheta(dx;q)\}\D_{q^{-1}}^n\left\{\bigg[
    \begin{array}{c}
         a_{1},\ldots,a_{r}  \\
         b_{1},\ldots,b_{s} 
    \end{array}
    ;q,x
    \bigg]_{\infty}\right\}
    \end{align*}
From Lemma     
    \begin{align*}
        &\E_{q}(y\D_{q^{-1}}|q^c)\left\{\vartheta(dx;q)\bigg[
    \begin{array}{c}
         a_{1},\ldots,a_{r}  \\
         b_{1},\ldots,b_{s} 
    \end{array}
    ;q,x
    \bigg]_{\infty}\right\}\\
    &=\sum_{n=0}^{\infty}\frac{y^n}{(q;q)_{n}}q^{(c-1)\binom{n}{2}}x^nd^n\vartheta(dx;q)\big[(-1)^nq^{\binom{n}{2}}\big]^{(1+s-r)}\left(-q^{s-r}x^{r-s-1}\frac{a_{1}\cdots a_{r}}{b_{1}\cdots b_{s}}\right)^n\nonumber\\
        &\hspace{0.5cm}\times\bigg[
    \begin{array}{c}
         q/a_{1},\ldots,q/a_{r}  \\
         q/b_{1},\ldots,q/b_{s}
    \end{array}
    ;q,x^{-1}
    \bigg]_{n}\bigg[
    \begin{array}{c}
         a_{1},\ldots,a_{r}  \\
         b_{1},\ldots,b_{s} 
    \end{array}
    ;q,x
    \bigg]_{\infty}\\
    &=\vartheta(dx;q)\bigg[
    \begin{array}{c}
         a_{1},\ldots,a_{r}  \\
         b_{1},\ldots,b_{s} 
    \end{array}
    ;q,x
    \bigg]_{\infty}\\
    &\hspace{1cm}\times\sum_{n=0}^{\infty}\big[(-1)^nq^{\binom{n}{2}}\big]^{(s-r+c)}\frac{1}{(q;q)_{n}}
        \bigg[
    \begin{array}{c}
         q/a_{1},\ldots,q/a_{r}  \\
         q/b_{1},\ldots,q/b_{s} 
    \end{array}
    ;q,x^{-1}
    \bigg]_{n}\left((-1)^cy\frac{q^{s-r}da_{1}\cdots a_{r}}{x^{s-r}b_{1}\cdots b_{s}}\right)^n\\
    &=\vartheta(dx;q)\bigg[
    \begin{array}{c}
         a_{1},\ldots,a_{r}  \\
         b_{1},\ldots,b_{s} 
    \end{array}
    ;q,x
    \bigg]_{\infty}{}_{r+1}\phi_{s+c}\left(
    \begin{array}{c}
         q/a_{1}x,\ldots,q/a_{r}x,0, \\
         q/b_{1}x,\ldots,q/b_{s}x,\mathbf{0}_{c}
    \end{array}
    ;q,(-1)^cy\frac{q^{s-r}da_{1}\cdots a_{r}}{x^{s-r}b_{1}\cdots b_{s}}
    \right).
    \end{align*}
\end{proof}

\begin{corollary}\label{coro13}
If $c\in\N$, then
\begin{equation}
    \E_{q}(y\D_{q^{-1}}|q^c)\{\vartheta(bx;q)(ax;q)_{\infty}\}
    =\vartheta(bx;q)(ax;q)_{\infty}\cdot{}_{2}\phi_{c}\left(
    \begin{array}{c}
         q/ax,0 \\
         \mathbf{0}_{c}
    \end{array}
    ;q, (-1)^cabxy/q
    \right).
\end{equation}
\end{corollary}

\begin{corollary}\label{coro14}
If $c\in\N$, then
\begin{equation}
    \E_{q}(y\D_{q^{-1}}|q^c)\left\{\frac{\vartheta(bx;q)}{(ax;q)_{\infty}}\right\}
    =\frac{\vartheta(bx;q)}{(ax;q)_{\infty}}\cdot{}_{1}\phi_{c+1}\left(
    \begin{array}{c}
         0 \\
         q/ax,\mathbf{0}_{c}
    \end{array}
    ;q, (-1)^cbqy/ax
    \right).
\end{equation}
\end{corollary}

\begin{corollary}
If $c\in\N$, then
\begin{multline}
    \E_{q}(y\D_{q^{-1}}|q^c)\left\{\vartheta(dx;q)\frac{(ax;q)_{\infty}}{(bx;q)_{\infty}}\right\}\\
    =\vartheta(dx;q)\frac{(ax;q)_{\infty}}{(bx;q)_{\infty}}\cdot{}_{2}\phi_{c+1}\left(
    \begin{array}{c}
         q/ax,0 \\
         q/bx,\mathbf{0}_{c}
    \end{array}
    ;q, (-1)^cady/b
    \right).
\end{multline}
\end{corollary}

\section{Summation formulas involving Jacobi's theta function}

\begin{theorem}\label{theo_BS_Ebq+}
For $b\geq2$
    \begin{equation}
        \sum_{n=-\infty}^{\infty}q^{\binom{n+1}{2}}(ax)^n\E_{b-1}(q^ny/x;q)=\vartheta(ax;q)\E_{b-2}(y/qax^2;q).
    \end{equation}
\end{theorem}
\begin{proof}
    \begin{align*}
        \sum_{n=-\infty}^{\infty}q^{\binom{n+1}{2}}(ax)^n\E_{b-1}(q^ny/x;q)&=\sum_{n=-\infty}^{\infty}q^{\binom{n+1}{2}}a^n\E_{q}(y\D_{q}|q^b)\{x^{n}\}\\
        &=\E_{q}(y\D_{q}|q^b)\left\{\sum_{n=-\infty}^{\infty}q^{\binom{n+1}{2}}(ax)^n\right\}\\
        &=\E_{q}(y\D_{q}|q^b)\left\{\vartheta(ax;q)\right\}\\
        &=\vartheta(ax;q)\E_{b-2}(y/qax^2;q)
    \end{align*}
\end{proof}

\begin{theorem}\label{theo_BS_Eb}
For $b\geq0$
    \begin{equation}
        \sum_{n=-\infty}^{\infty}q^{\binom{n+1}{2}}(ax)^n\E_{b+1}(q^{-n}y/x;q)=\vartheta(ax;q)\E_{b}(ay;q).
    \end{equation}
\end{theorem}
\begin{proof}
    \begin{align*}
        \sum_{n=-\infty}^{\infty}q^{\binom{n+1}{2}}(ax)^n\E_{b+1}(q^{-n}y/x;q)&=\sum_{n=-\infty}^{\infty}q^{\binom{n}{2}}a^n\E_{q}(y\D_{q^{-1}}|q^b)\{x^n\}\\
        &=\E_{q}(y\D_{q^{-1}}|q^b)\left\{\sum_{n=-\infty}^{\infty}q^{\binom{n+1}{2}}(ax)^n\right\}\\
        &=\E_{q}(y\D_{q^{-1}}|q^b)\{\vartheta(ax;q)\}\\
        &=\vartheta(ax;q)\E_{b}(ay;q).
    \end{align*}
\end{proof}

By making $b=1$ in Theorem \ref{theo_BS_Eb} we obtain the following result. 
\begin{corollary}
    \begin{equation}
        \sum_{n=-\infty}^{\infty}q^{\binom{n+1}{2}}(ax)^n\K_{\infty}(q^{-n}y/x)=\vartheta(ax;q)(-ay;q)_{\infty}.
    \end{equation}
\end{corollary}

\begin{theorem}
    \begin{equation}
        \sum_{n=-\infty}^{\infty}q^{\binom{n+1}{2}}(cx)^n\cdot{}_{1}\phi_{1}\left(
    \begin{array}{c}
         qx/a \\
         0
    \end{array}
    ;q, q^{n-1}ay/x^2
    \right)
    =\vartheta(cx;q)\frac{(-y/qcx^2;q)_{\infty}}{(-ay/q^2cx^3;q)_{\infty}}.
    \end{equation}
\end{theorem}
\begin{proof}
    \begin{align*}
        &\sum_{n=-\infty}^{\infty}q^{\binom{n+1}{2}}c^{n}x^{n}(a/x;q)_{\infty}\cdot{}_{1}\phi_{1}\left(
    \begin{array}{c}
         qx/a \\
         0
    \end{array}
    ;q, q^{n-1}ay/x^2
    \right)\\
        &\hspace{1cm}=\sum_{n=-\infty}^{\infty}q^{\binom{n+1}{2}}c^n\E_{q}(y\D_{q}|q^3)\{x^n(a/x;q)_{\infty}\}\\
        &\hspace{1cm}=\E_{q}(y\D_{q}|q^3)\left\{(a/x;q)_{\infty}\vartheta(cx;q)\right\}\\
        &\hspace{1cm}=(a/x;q)_{\infty}\vartheta(cx;q)\frac{(-y/qcx^2;q)_{\infty}}{(-ay/q^2cx^3;q)_{\infty}}.
    \end{align*}
\end{proof}

\begin{theorem}
For $b\geq4$,
    \begin{multline}
        \sum_{n=-\infty}^{\infty}q^{\binom{n+1}{2}}(cx)^n\cdot{}_{1}\phi_{b-2}\left(
    \begin{array}{c}
         qx/a \\
         \mathbf{0}_{b-2}
    \end{array}
    ;q, (-1)^{b-1}q^{n-1}ay/x^2
    \right)\\
    =\vartheta(cx;q){}_{1}\phi_{b-3}\left(
    \begin{array}{c}
         qx/a \\
         \mathbf{0}_{b-3}
    \end{array}
    ;q, (-1)^{b-2}ay/cx^3
    \right).
    \end{multline}
\end{theorem}
\begin{proof}
From Theorems 4 and 5
    \begin{align*}
        &\sum_{n=-\infty}^{\infty}q^{\binom{n+1}{2}}c^{n}x^{n}(a/x;q)_{\infty}\cdot{}_{1}\phi_{b-2}\left(
    \begin{array}{c}
         qx/a \\
         \mathbf{0}_{b-2}
    \end{array}
    ;q, (-1)^{b-1}q^{n-1}ay/x^2
    \right)\\
        &\hspace{1cm}=\sum_{n=-\infty}^{\infty}q^{\binom{n+1}{2}}c^n\E_{q}(y\D_{q}|q^b)\{x^n(a/x;q)_{\infty}\}\\
        &\hspace{1cm}=\E_{q}(y\D_{q}|q^b)\left\{(a/x;q)_{\infty}\vartheta(cx;q)\right\}\\
        &\hspace{1cm}=(a/x;q)_{\infty}\vartheta(cx;q){}_{1}\phi_{b-3}\left(
    \begin{array}{c}
         qx/a \\
         \mathbf{0}_{b-3}
    \end{array}
    ;q, (-1)^{b-2}ay/cx^3
    \right).
    \end{align*}
\end{proof}

\begin{theorem}
For $b\geq1$,
    \begin{multline}
        \sum_{n=-\infty}^{\infty}q^{\binom{n+1}{2}}(cx)^n{}_{1}\phi_{b+1}\left(
    \begin{array}{c}
         0 \\
         a/x,\mathbf{0}_{b}
    \end{array}
    ;q, (-1)^{b+1}ay/q^{n}x
    \right)=\\
    \vartheta(cx;q){}_{1}\phi_{b}\left(
    \begin{array}{c}
         0 \\
         a/x,\mathbf{0}_{b-1}
    \end{array}
    ;q, (-1)^{b}cy
    \right).
    \end{multline}
\end{theorem}
\begin{proof}
    \begin{align*}
        &\sum_{n=-\infty}^{\infty}q^{\binom{n+1}{2}}c^{n}x^n(a/x;q)_{\infty}\cdot{}_{1}\phi_{b+1}\left(
    \begin{array}{c}
         0 \\
         a/x,\mathbf{0}_{b}
    \end{array}
    ;q, (-1)^{b+1}ay/q^{n}x
    \right)\\
        &\hspace{1cm}=\sum_{n=-\infty}^{\infty}q^{\binom{n+1}{2}}c^n\E_{q}(y\D_{q^{-1}}|q^b)\{x^n(a/x;q)_{\infty}\}\\
        &\hspace{1cm}=\E_{q}(y\D_{q^{-1}}|q^b)\left\{(a/x;q)_{\infty}\vartheta(cx;q)\right\}\\
        &\hspace{1cm}=(a/x;q)_{\infty}\vartheta(cx;q){}_{1}\phi_{b}\left(
    \begin{array}{c}
         0 \\
         a/x,\mathbf{0}_{b-1}
    \end{array}
    ;q, (-1)^{b}cy
    \right).
    \end{align*}
\end{proof}

\begin{theorem}
    \begin{equation}
        \sum_{n=-\infty}^{\infty}q^{\binom{n+1}{2}}(dx)^n\cdot{}_{1}\phi_{1}\left(
    \begin{array}{c}
         x \\
         qax/b
    \end{array}
    ;q,-q^{n+1}y
    \right)
    =\vartheta(dx;q)\cdot{}_{2}\phi_{1}\left(
    \begin{array}{c}
         x,0 \\
         qax/b
    \end{array}
    ;q,y/dx
    \right)
    \end{equation}
\end{theorem}
\begin{proof}
    \begin{align*}
    &\sum_{n=-\infty}^{\infty}q^{\binom{n+1}{2}}(dx)^n\frac{(ax,q/ax;q)_{\infty}}{(x,b/ax;q)_{\infty}}\cdot{}_{1}\phi_{1}\left(
    \begin{array}{c}
         x \\
         qax/b
    \end{array}
    ;q,-q^{n+1}y
    \right)\\
    &\hspace{1cm}=\sum_{n=-\infty}^{\infty}q^{\binom{n+1}{2}}d^n\E_{q}(y\D_{q}|q^2)\left\{x^n\frac{(ax,q/ax;q)_{\infty}}{(x,b/ax;q)_{\infty}}\right\}\\
    &\hspace{1cm}=\E_{q}(y\D_{q}|q^2)\left\{\frac{(ax,q/ax;q)_{\infty}}{(x,b/ax;q)_{\infty}}\sum_{n=-\infty}^{\infty}q^{\binom{n+1}{2}}(dx)^n\right\}\\
    &\hspace{1cm}=\E_{q}(y\D_{q}|q^2)\bigg\{\vartheta(dx;q)\frac{(ax,q/ax;q)_{\infty}}{(x,b/ax;q)_{\infty}}\bigg\}\\
    &\hspace{1cm}=\vartheta(dx;q)\frac{(ax,q/ax;q)_{\infty}}{(x,b/ax;q)_{\infty}}\cdot{}_{2}\phi_{1}\left(
    \begin{array}{c}
         x,0 \\
         qax/b
    \end{array}
    ;q,y/dx
    \right)
    \end{align*}
\end{proof}

\begin{theorem}
For $c\geq3$
    \begin{multline}
        \sum_{n=-\infty}^{\infty}q^{\binom{n+1}{2}}(cx)^n{}_{1}\phi_{c-1}\left(
    \begin{array}{c}
         x \\
         qax/b,\mathbf{0}_{c-2}
    \end{array}
    ;q,(-1)^{c-1}q^{n+1}y
    \right)\\
    =\vartheta(dx;q)\cdot{}_{1}\phi_{c-2}\left(
    \begin{array}{c}
         x \\
         qax/b,\mathbf{0}_{c-3}
    \end{array}
    ;q,(-1)^{c-2}y/dx
    \right).
    \end{multline}
\end{theorem}
\begin{proof}
From Theorems 6 and 7
\begin{align*}
    &\sum_{n=-\infty}^{\infty}q^{\binom{n+1}{2}}(dx)^n\frac{(ax,q/ax;q)_{\infty}}{(x,b/ax;q)_{\infty}}\cdot{}_{1}\phi_{c-1}\left(
    \begin{array}{c}
         x \\
         qax/b,\mathbf{0}_{c-2}
    \end{array}
    ;q,(-1)^{c-1}q^{n+1}y
    \right)\\
    &\hspace{1cm}=\sum_{n=-\infty}^{\infty}q^{\binom{n+1}{2}}d^n\E_{q}(y\D_{q}|q^c)\left\{x^n\frac{(ax,q/ax;q)_{\infty}}{(x,b/ax;q)_{\infty}}\right\}\\
    &\hspace{1cm}=\E_{q}(y\D_{q}|q^c)\left\{\frac{(ax,q/ax;q)_{\infty}}{(x,b/ax;q)_{\infty}}\sum_{n=-\infty}^{\infty}q^{\binom{n+1}{2}}(dx)^n\right\}\\
    &\hspace{1cm}=\E_{q}(y\D_{q}|q^c)\left\{\vartheta(dx;q)\frac{(ax,q/ax;q)_{\infty}}{(x,b/ax;q)_{\infty}}\right\}\\
    &\hspace{1cm}=\vartheta(dx;q)\frac{(ax,q/ax;q)_{\infty}}{(x,b/ax;q)_{\infty}}\cdot{}_{1}\phi_{c-2}\left(
    \begin{array}{c}
         x \\
         qax/b,\mathbf{0}_{c-3}
    \end{array}
    ;q,(-1)^{c-2}y/dx
    \right)
\end{align*}
\end{proof}

\begin{theorem}
    \begin{equation}
        \sum_{n=-\infty}^{\infty}q^{\binom{n+1}{2}}(dx)^n\cdot{}_{1}\phi_{1}\left(
    \begin{array}{c}
         b/ax \\
         q/x
    \end{array}
    ;q,-aq^{n}y/x
    \right)
    =\vartheta(dx;q)\cdot{}_{2}\phi_{1}\left(
    \begin{array}{c}
         b/ax,0 \\
         q/x
    \end{array}
    ;q,ady
    \right).
    \end{equation}
\end{theorem}
\begin{proof}
From Theorems 8 and 9
    \begin{align*}
        &\sum_{n=-\infty}^{\infty}q^{\binom{n+1}{2}}(dx)^n\frac{(ax,q/ax;q)_{\infty}}{(x,b/ax;q)_{\infty}}\cdot{}_{1}\phi_{1}\left(
    \begin{array}{c}
         b/ax \\
         q/x
    \end{array}
    ;q,-aq^{n}y/x
    \right)\\
    &\hspace{1cm}=\sum_{n=-\infty}^{\infty}q^{\binom{n+1}{2}}d^n\E_{q}(y\D_{q^{-1}}|1)\left\{x^n\frac{(ax,q/ax;q)_{\infty}}{(x,b/ax;q)_{\infty}}\right\}\\
    &\hspace{1cm}=\E_{q}(y\D_{q^{-1}}|1)\left\{\frac{(ax,q/ax;q)_{\infty}}{(x,b/ax;q)_{\infty}}\sum_{n=-\infty}^{\infty}q^{\binom{n+1}{2}}(dx)^n\right\}\\
    &\hspace{1cm}=\E_{q}(y\D_{q^{-1}}|1)\left\{\frac{(ax,q/ax;q)_{\infty}}{(x,b/ax;q)_{\infty}}\vartheta(dx;q)\right\}\\
    &\hspace{1cm}=\vartheta(dx;q)\frac{(ax,q/ax;q)_{\infty}}{(x,b/ax;q)_{\infty}}\cdot{}_{2}\phi_{1}\left(
    \begin{array}{c}
         b/ax,0 \\
         q/x
    \end{array}
    ;q,ady
    \right).
    \end{align*}
\end{proof}

\begin{theorem}
For $c\geq1$
    \begin{multline}
        \sum_{n=-\infty}^{\infty}q^{\binom{n+1}{2}}(dx)^n\cdot{}_{1}\phi_{c+1}\left(
    \begin{array}{c}
         b/ax \\
         q/x,\mathbf{0}_{c}
    \end{array}
    ;q,(-1)^{c+1}aq^{n}y/x
    \right)\\
    =\vartheta(dx;q)\cdot{}_{1}\phi_{c}\left(
    \begin{array}{c}
         b/ax \\
         q/x,\mathbf{0}_{c-1}
    \end{array}
    ;q,(-1)^{c}ady
    \right).
    \end{multline}
\end{theorem}
\begin{proof}
From Theorems 8 and 9
    \begin{align*}
        &\sum_{n=-\infty}^{\infty}q^{\binom{n+1}{2}}(dx)^n\frac{(ax,q/ax;q)_{\infty}}{(x,b/ax;q)_{\infty}}\cdot{}_{1}\phi_{c+1}\left(
    \begin{array}{c}
         b/ax \\
         q/x\mathbf{0}_{c}
    \end{array}
    ;q,(-1)^{c+1}aq^{n}y/x
    \right)\\
    &\hspace{1cm}=\sum_{n=-\infty}^{\infty}q^{\binom{n+1}{2}}d^n\E_{q}(y\D_{q^{-1}}|q^c)\left\{x^n\frac{(ax,q/ax;q)_{\infty}}{(x,b/ax;q)_{\infty}}\right\}\\
    &\hspace{1cm}=\E_{q}(y\D_{q^{-1}}|q^c)\left\{\frac{(ax,q/ax;q)_{\infty}}{(x,b/ax;q)_{\infty}}\sum_{n=-\infty}^{\infty}q^{\binom{n+1}{2}}(dx)^n\right\}\\
    &\hspace{1cm}=\E_{q}(y\D_{q^{-1}}|q^c)\left\{\frac{(ax,q/ax;q)_{\infty}}{(x,b/ax;q)_{\infty}}\vartheta(dx;q)\right\}\\
    &\hspace{1cm}=\vartheta(dx;q)\frac{(ax,q/ax;q)_{\infty}}{(x,b/ax;q)_{\infty}}\cdot{}_{1}\phi_{c}\left(
    \begin{array}{c}
         b/ax \\
         q/x,\mathbf{0}_{c-1}
    \end{array}
    ;q,(-1)^{c}ady
    \right).
    \end{align*}
\end{proof}

\begin{theorem}
For $c\geq2$
    \begin{multline}
        \sum_{n=-\infty}^{\infty}q^{\binom{n+1}{2}}(dx)^n{}_{r+1}\phi_{r+c-1}\left(
    \begin{array}{c}
         b_{1}x,\ldots,b_{r+1}x, \\
         a_{1}x,\ldots,a_{r+c-1}x
    \end{array}
    ;q,(-1)^{c-1}q^ny/x
    \right)\\
    =\vartheta(dx;q){}_{r+1}\phi_{r+c-2}\left(
    \begin{array}{c}
         b_{1}x,\ldots,b_{r+1}x, \\
         a_{1}x,\ldots,a_{r+c-2}x
    \end{array}
    ;q,(-1)^{c-2}y/qdx^2
    \right).
    \end{multline}
\end{theorem}
\begin{proof}
From Theorems 10 and 11
    \begin{align*}
        &\sum_{n=-\infty}^{\infty}q^{\binom{n+1}{2}}(dx)^n\bigg[
    \begin{array}{c}
         a_{1},\ldots,a_{r+c-1}  \\
         b_{1},\ldots,b_{r+1} 
    \end{array}
    ;q,x
    \bigg]_{\infty}{}_{r+1}\phi_{r+c-1}\left(
    \begin{array}{c}
         b_{1}x,\ldots,b_{r+1}x, \\
         a_{1}x,\ldots,a_{r+c-1}x
    \end{array}
    ;q,(-1)^{c-1}q^ny/x
    \right)\\
    &=\sum_{n=-\infty}^{\infty}q^{\binom{n+1}{2}}d^n\E_{q}(y\D_{q}|q^c)\left\{x^n\bigg[
    \begin{array}{c}
         a_{1},\ldots,a_{r+c-1}  \\
         b_{1},\ldots,b_{r+1} 
    \end{array}
    ;q,x
    \bigg]_{\infty}\right\}\\
    &=\E_{q}(y\D_{q}|q^c)\left\{\bigg[
    \begin{array}{c}
         a_{1},\ldots,a_{r+c-1}  \\
         b_{1},\ldots,b_{r+1} 
    \end{array}
    ;q,x
    \bigg]_{\infty}\sum_{n=-\infty}^{\infty}q^{\binom{n+1}{2}}(dx)^n\right\}\\
    &=\E_{q}(y\D_{q}|q^c)\left\{\bigg[
    \begin{array}{c}
         a_{1},\ldots,a_{r+c-1}  \\
         b_{1},\ldots,b_{r+1} 
    \end{array}
    ;q,x
    \bigg]_{\infty}\vartheta(dx;q)\right\}\\
    &=\vartheta(dx;q)\bigg[
    \begin{array}{c}
         a_{1},\ldots,a_{r+c-2}  \\
         b_{1},\ldots,b_{r+1} 
    \end{array}
    ;q,x
    \bigg]_{\infty}{}_{r+1}\phi_{r+c-2}\left(
    \begin{array}{c}
         b_{1}x,\ldots,b_{r+1}x, \\
         a_{1}x,\ldots,a_{r+c-2}x
    \end{array}
    ;q,(-1)^{c-2}y/qdx^2
    \right).
    \end{align*}
\end{proof}

\begin{theorem}
For $c\in\N$
    \begin{multline}
        \sum_{n=-\infty}^{\infty}q^{\binom{n+1}{2}}(dx)^n{}_{r}\phi_{s+c}\left(
    \begin{array}{c}
         q/a_{1}x,\ldots,q/a_{r}x \\
         q/b_{1}x,\ldots,q/b_{s}x,\mathbf{0}_{c}
    \end{array}
    ;q,(-1)^{c+1}\frac{q^{s-r-n}ya_{1}\cdots a_{r}}{x^{1+s-r}b_{1}\cdots b_{s}}
    \right)\\
    =\vartheta(dx;q){}_{r+1}\phi_{s+c}\left(
    \begin{array}{c}
         q/a_{1}x,\ldots,q/a_{r}x,0, \\
         q/b_{1}x,\ldots,q/b_{s}x,\mathbf{0}_{c}
    \end{array}
    ;q,(-1)^cy\frac{q^{s-r}a_{1}\cdots a_{r}}{x^{s-r}b_{1}\cdots b_{s}}
    \right).
    \end{multline}
\end{theorem}
\begin{proof}
    \begin{align*}
        &\sum_{n=-\infty}^{\infty}q^{\binom{n+1}{2}}d^nx^n\bigg[
    \begin{array}{c}
         a_{1},\ldots,a_{r}  \\
         b_{1},\ldots,b_{s} 
    \end{array}
    ;q,x
    \bigg]_{\infty}\\
    &\hspace{3cm}{}_{r}\phi_{s+c}\left(
    \begin{array}{c}
         q/a_{1}x,\ldots,q/a_{r}x \\
         q/b_{1}x,\ldots,q/b_{s}x,\mathbf{0}_{c}
    \end{array}
    ;q,(-1)^{c+1}\frac{q^{s-r-n}ya_{1}\cdots a_{r}}{x^{1+s-r}b_{1}\cdots b_{s}}
    \right)\\
    &=\sum_{n=-\infty}^{\infty}q^{\binom{n+1}{2}}d^n\E_{q}(y\D_{q^{-1}}|q^c)\left\{x^n\bigg[
    \begin{array}{c}
         a_{1},\ldots,a_{r}  \\
         b_{1},\ldots,b_{s} 
    \end{array}
    ;q,x
    \bigg]_{\infty}\right\}\\
    &=\E_{q}(y\D_{q^{-1}}|q^c)\left\{\bigg[
    \begin{array}{c}
         a_{1},\ldots,a_{r}  \\
         b_{1},\ldots,b_{s} 
    \end{array}
    ;q,x
    \bigg]_{\infty}\sum_{n=-\infty}^{\infty}q^{\binom{n}{2}}d^nx^n\right\}\\
    &=\E_{q}(y\D_{q^{-1}}|q^c)\left\{\vartheta(dx;q)\bigg[
    \begin{array}{c}
         a_{1},\ldots,a_{r}  \\
         b_{1},\ldots,b_{s} 
    \end{array}
    ;q,x
    \bigg]_{\infty}\right\}\\
    &=\vartheta(dx;q)\bigg[
    \begin{array}{c}
         a_{1},\ldots,a_{r}  \\
         b_{1},\ldots,b_{s} 
    \end{array}
    ;q,x
    \bigg]_{\infty}{}_{r+1}\phi_{s+c}\left(
    \begin{array}{c}
         q/a_{1}x,\ldots,q/a_{r}x,0, \\
         q/b_{1}x,\ldots,q/b_{s}x,\mathbf{0}_{c}
    \end{array}
    ;q,(-1)^cy\frac{q^{s-r}a_{1}\cdots a_{r}}{x^{s-r}b_{1}\cdots b_{s}}
    \right).
    \end{align*}
\end{proof}

\section{Application to bilateral basic hypergeometric series}

\subsection{Summations involving Jacobi theta function}

\begin{theorem}
    \begin{equation}
        {}_{0}\psi_{1}\left[
    \begin{array}{c}
         -\\
         b
    \end{array}
    ;q,qax
    \right]=\frac{\vartheta(-ax;q)}{(b,-b/qax;q)_{\infty}}.
    \end{equation}
\end{theorem}
\begin{proof}
By Theorem \ref{theo_BS_Ebq+} with $b=2$,
    \begin{align*}
        \sum_{n=-\infty}^{\infty}q^{\binom{n+1}{2}}(-ax)^n\frac{(y/x;q)_{\infty}}{(y/x;q)_{n}}&=\sum_{n=-\infty}^{\infty}q^{\binom{n+1}{2}}(-ax)^n\E_{2}(q^{-n}y/x;q)\\
        &=\vartheta(-ax;q)\E_{0}(-y/qax^2;q)\\
        &=\frac{\vartheta(-ax;q)}{(-y/qax^2;q)_{\infty}}.
    \end{align*}
Then
\begin{equation}
        \sum_{n=-\infty}^{\infty}q^{\binom{n+1}{2}}\frac{(-ax)^n}{(y/x;q)_{n}}=\frac{\vartheta(ax;q)}{(y/x;q)_{\infty}(-y/qax^2;q)_{\infty}}.
    \end{equation}  
Finally, set $y/x=b$.
\end{proof}

\begin{theorem}
    \begin{equation}
        {}_{1}\psi_{1}\left[
    \begin{array}{c}
         -qx/y\\
         0
    \end{array}
    ;q,axy/q
    \right]=\frac{\vartheta(ax;q)}{(ay,-y/x;q)_{\infty}}.
    \end{equation}
\end{theorem}
\begin{proof}
By Theorem \ref{theo_BS_Eb} with $b=0$,
    \begin{align*}
        &\sum_{n=-\infty}^{\infty}(axy/q)^n(-qx/y;q)_{n}(-y/x;q)_{\infty}\\
        &\hspace{1cm}=\sum_{n=-\infty}^{\infty}q^{\binom{n+1}{2}}(ax/q)^nq^{-\binom{n+1}{2}}y^n(-qx/y;q)_{n}(-y/x;q)_{\infty}\\
        &\hspace{1cm}=\sum_{n=-\infty}^{\infty}q^{\binom{n+1}{2}}(ax/q)^n\E_{1}(q^{-n}y/x;q)\\
        &\hspace{1cm}=\vartheta(ax/q;q)\E_{0}(ay;q)=\frac{\vartheta(ax/q;q)}{(ay;q)_{\infty}}.
    \end{align*}
Then
\begin{equation}
        \sum_{n=-\infty}^{\infty}(axy/q)^n(-qx/y;q)_{n}=\frac{\vartheta(ax/q;q)}{(-y/x;q)_{\infty}(ay;q)_{\infty}}.
    \end{equation}
\end{proof}

\begin{theorem}
    \begin{equation}
        {}_{1}\psi_{2}\left[
    \begin{array}{c}
         qx/y\\
         q^2/ay,0
    \end{array}
    ;q,-q^2bx/a
    \right]=\vartheta(bx;q)\frac{(ay/q;q)_{\infty}}{(y/x;q)_{\infty}}{}_2\phi_{0}\left(\begin{array}{c}
         q/ax,0\\
         -
    \end{array};q,-abxy/q\right).
    \end{equation}
\end{theorem}
\begin{proof}
From Corollary \ref{coro13} with $c=0$, we have that  
    \begin{align*}
        &\sum_{n=-\infty}^{\infty}q^{\binom{n+1}{2}}\frac{(qx/y;q)_{n}}{(q^2/ay;q)_{n}}\frac{(ax,y/x;q)_{\infty}}{(-ay/q;q)_{\infty}}(bqx/a)^n\\
        &\hspace{1cm}=\sum_{n=-\infty}^{\infty}q^{\binom{n+1}{2}}b^n\E_{q}(-y\D_{q^{-1}}|1)\{x^n(ax;q)_{\infty}\}\\
        &\hspace{1cm}=\E_{q}(-y\D_{q^{-1}}|1)\{\vartheta(bx;q)(ax;q)_{\infty}\}\\
        &\hspace{1cm}=\vartheta(bx;q)(ax;q)_{\infty}\cdot{}_2\phi_{0}\left(\begin{array}{c}
         q/ax,0\\
         -
    \end{array};q,-abxy/a\right).
    \end{align*}
Then
\begin{align*}
    &\sum_{n=-\infty}^{\infty}q^{\binom{n+1}{2}}\frac{(qx/y;q)_{n}}{(q^2/ay;q)_{n}}(bq^2x/a)^n\\
    &\hspace{2cm}=\vartheta(bx;q)\frac{(ay/q;q)_{\infty}}{(y/x;q)_{\infty}}{}_2\phi_{0}\left(\begin{array}{c}
         q/ax,0\\
         -
    \end{array};q,-abxy/q\right).
\end{align*}
\end{proof}

\subsection{Summations involving Ramanujan summation}

\begin{theorem}
    \begin{equation}
        {}_{2}\psi_{2}\left[
    \begin{array}{c}
         a,y/x\\
         b,0
    \end{array}
    ;q,dx
    \right]=\frac{(y/x,q,b/a,adx,q/adx;q)_{\infty}}{(b,q/a,dx,b/adx;q)_{\infty}}{}_2\phi_{1}\left(\begin{array}{c}
         x,0\\
         qax/b
    \end{array};q,qy\right).
    \end{equation}
\end{theorem}
\begin{proof}
    \begin{align*}
        &\frac{1}{(y/x;q)_{\infty}}\sum_{n=-\infty}^{\infty}\frac{(a;q)_{n}(y/x;q)_{n}}{(b;q)_{n}}(dx)^n\\
        &\hspace{1cm}=\sum_{n=-\infty}^{\infty}\frac{(a;q)_{n}}{(b;q)_{n}}(dx)^n\frac{1}{(q^ny/x;q)_{\infty}}\\
        &\hspace{1cm}=\sum_{n=-\infty}^{\infty}\frac{(a;q)_{n}}{(b;q)_{n}}(dx)^n\E_{0}(q^ny/x;q)\\
        &\hspace{1cm}=\sum_{n=-\infty}^{\infty}\frac{(a;q)_{n}}{(b;q)_{n}}\E_{q}(y\D_{q}|q)\left\{(dx)^n\right\}\\
        &\hspace{1cm}=\E_{q}(y\D_{q}|q)\left\{\sum_{n=-\infty}^{\infty}\frac{(a;q)_{n}}{(b;q)_{n}}(dx)^n\right\}\\
        &\hspace{1cm}=\E_{q}(y\D_{q}|q)\left\{\frac{(q,b/a,adx,q/adx;q)_{\infty}}{(b,q/a,dx,b/adx;q)_{\infty}}\right\}\\
        &\hspace{1cm}=\frac{(q,b/a;q)_{\infty}}{(b,q/a;q)_{\infty}}\frac{(adx,q/adx;q)_{\infty}}{(dx,b/adx;q)_{\infty}}{}_{2}\phi_{1}\left(
    \begin{array}{c}
         x,0 \\
         qadx/b
    \end{array}
    ;q,qy
    \right).
    \end{align*}
Then
\begin{equation*}
    \sum_{n=-\infty}^{\infty}\frac{(a;q)_{n}(y/x;q)_{n}}{(b;q)_{n}}(dx)^n=\frac{(y/x,q,b/a,adx,q/adx;q)_{\infty}}{(b,q/adx,b/adx;q)_{\infty}}{}_{2}\phi_{1}\left(
    \begin{array}{c}
         x,0 \\
         qadx/b
    \end{array}
    ;q,qy
    \right).
\end{equation*}
\end{proof}

\begin{theorem}
    \begin{equation}
        {}_{2}\psi_{2}\left[
    \begin{array}{c}
         a,y/x\\
         b,ay/x
    \end{array}
    ;q,dx
    \right]=\frac{(y/x,q,b/a,adx,q/adx;q)_{\infty}}{(ay/x,b,q/a,dx,b/adx;q)_{\infty}}{}_2\phi_{1}\left(\begin{array}{c}
         ax,dx\\
         qadx/b
    \end{array};q,qy\right).
    \end{equation}
\end{theorem}
\begin{proof}
From Corollary \ref{coro11} and Ramanujan's summation formula Eq.(\ref{eqn_ramanujan}), we have
    \begin{align*}
        &\frac{(ay/x;q)_{\infty}}{(ax,y/x;q)_{\infty}}\sum_{n=-\infty}^{\infty}\frac{(a;q)_{n}(y/x;q)_{n}}{(b;q)_{n}(ay/x;q)_{n}}(dx)^n\\
        &=\sum_{n=-\infty}^{\infty}\frac{(a;q)_{n}}{(b;q)_{n}}\E_{q}(y\D_{q}|q)\left\{\frac{(dx)^n}{(ax;q)_{\infty}}\right\}\\
        &=\E_{q}(y\D_{q}|q)\left\{\frac{1}{(ax;q)_{\infty}}\sum_{n=-\infty}^{\infty}\frac{(a;q)_{n}}{(b;q)_{n}}(dx)^n\right\}\\
        &=\E_{q}(y\D_{q}|q)\left\{\frac{1}{(ax;q)_{\infty}}\frac{(q,b/a,adx,q/adx;q)_{\infty}}{(b,q/a,dx,b/adx;q)_{\infty}}\right\}\\
        &=\frac{(q,b/a;q)_{\infty}}{(b,q/a;q)_{\infty}}\E_{q}(y\D_{q}|q)\left\{\frac{1}{(ax;q)_{\infty}}\frac{(adx,q/adx;q)_{\infty}}{(dx,b/adx;q)_{\infty}}\right\}.
    \end{align*}
By Definition \ref{def_PTO} and Theorem \ref{theo6}
    \begin{align*}
        &\E_{q}(y\D_{q}|q)\left\{\frac{1}{(ax;q)_{\infty}}\frac{(adx,q/adx;q)_{\infty}}{(dx,b/adx;q)_{\infty}}\right\}\\
        &=\sum_{n=0}^{\infty}q^{\binom{n}{2}}\frac{y^n}{(q;q)_{n}}\frac{1}{(q^nax;q)_{\infty}}\frac{q^n}{q^{\binom{n}{2}}}\frac{(dx;q)_{n}}{(qadx/b;q)_{n}}\frac{(adx,q/adx;q)_{\infty}}{(dx,b/adx;q)_{\infty}}\\
        &=\frac{(adx,q/adx;q)_{\infty}}{(ax,dx,b/adx;q)_{\infty}}\sum_{n=0}^{\infty}\frac{(qy)^n}{(q;q)_{n}}\frac{(ax;q)_{n}(dx;q)_{n}}{(qadx/b;q)_{n}}\\
        &=\frac{(adx,q/adx;q)_{\infty}}{(ax,dx,b/adx;q)_{\infty}}{}_{2}\phi_{1}\left(
    \begin{array}{c}
         ax,dx \\
         qadx/b
    \end{array}
    ;q,qy
    \right).
    \end{align*}
Then
\begin{equation*}
    \sum_{n=-\infty}^{\infty}\frac{(a;q)_{n}(y/x;q)_{n}}{(b;q)_{n}(ay/x;q)_{n}}(dx)^n=\frac{(y/x,q,b/a,adx,q/adx;q)_{\infty}}{(ay/x,b,q/a,dx,b/adx;q)_{\infty}}{}_{2}\phi_{1}\left(
    \begin{array}{c}
         ax,dx \\
         qadx/b
    \end{array}
    ;q,qy
    \right).   
\end{equation*}
\end{proof}

\begin{theorem}
    \begin{equation}
        {}_{2}\psi_{2}\left[
    \begin{array}{c}
         a,-qy/x\\
         b,-q^2/ay
    \end{array}
    ;q,\frac{qcx}{a}
    \right]=\frac{(-ay/q,q,b/a,acx,q/acx;q)_{\infty}}{(-y/x,b,q/a,cx,b/acx;q)_{\infty}}{}_2\phi_{1}\left(\begin{array}{c}
         q/ax,b/acx\\
         q/cx
    \end{array};q,-\frac{a^2y}{q}\right).
    \end{equation}
\end{theorem}
\begin{proof}
From Corollary 10
    \begin{align*}
        &\sum_{n=-\infty}^{\infty}\frac{(a;q)_{n}}{(b;q)_{n}}(cqx/a)^n\frac{(-qx/y;q)_{n}}{(-q^2/ay;q)_{n}}\frac{(ax,-y/x;q)_{\infty}}{(-ay/q;q)_{\infty}}\\
        &\hspace{1cm}=\sum_{n=-\infty}^{\infty}\frac{(a;q)_{n}}{(b;q)_{n}}c^n\E_{q}(y\D_{q^{-1}}|1)\left\{x^n(ax;q)_{\infty}\right\}\\
        &\hspace{1cm}=\E_{q}(y\D_{q^{-1}}|1)\left\{(ax;q)_{\infty}\sum_{n=-\infty}^{\infty}\frac{(a;q)_{n}}{(b;q)_{n}}(cx)^n\right\}\\
        &\hspace{1cm}=\frac{(q,b/a;q)_{\infty}}{(b,q/a;q)_{\infty}}\E_{q}(y\D_{q^{-1}}|1)\left\{(ax;q)_{\infty}\frac{(acx,q/acx;q)_{\infty}}{(cx,b/acx;q)_{\infty}}\right\}.
    \end{align*}
From Definition \ref{def_PTO} and Theorem \ref{theo6}
\begin{align*}
    &\E_{q}(y\D_{q^{-1}}|1)\left\{(ax;q)_{\infty}\frac{(acx,q/acx;q)_{\infty}}{(cx,b/acx;q)_{\infty}}\right\}\\
    &\hspace{1cm}=\sum_{n=0}^{\infty}\frac{y^n}{(q;q)_{n}}(q^{-n}ax;q)_{\infty}\frac{q^{\binom{n}{2}}}{x^n}\frac{a^n(b/acx;q)_{n}}{(q/cx;q)_{n}}\frac{(acx,q/acx;q)_{\infty}}{(cx,b/acx;q)_{\infty}}\\    
    &\hspace{1cm}=\frac{(acx,q/acx;q)_{\infty}}{(cx,b/acx;q)_{\infty}}\sum_{n=0}^{\infty}\frac{y^n}{(q;q)_{n}}\frac{(-ax)^n}{q^{\binom{n+1}{2}}}(q/ax;q)_{n}(ax;q)_{\infty}\frac{q^{\binom{n}{2}}}{x^n}\frac{a^n(b/acx;q)_{n}}{(q/cx;q)_{n}}\\
    &\hspace{1cm}=\frac{(ax,acx,q/acx;q)_{\infty}}{(cx,b/acx;q)_{\infty}}\sum_{n=0}^{\infty}\frac{(q/ax;q)_{n}(b/acx;q)_{n}}{(q;q)_{n}(q/cx;q)_{n}}(-a^2y/q)^n\\    
    &\hspace{1cm}=\frac{(ax,acx,q/acx;q)_{\infty}}{(cx,b/acx;q)_{\infty}}{}_2\phi_{1}\left(\begin{array}{c}
         q/ax,b/acx\\
         q/cx
    \end{array};q,-\frac{a^2y}{q}\right).
\end{align*}
Then 
\begin{align*}
    &\sum_{n=-\infty}^{\infty}\frac{(a;q)_{n}}{(b;q)_{n}}\frac{(-qx/y;q)_{n}}{(-q^2/ay;q)_{n}}(cqx/a)^n\\
    &\hspace{1cm}=\frac{(-ay/q,q,b/a,acx,q/acx;q)_{\infty}}{(-y/x,b,q/a,cx,b/acx;q)_{\infty}}{}_2\phi_{1}\left(\begin{array}{c}
         q/ax,b/acx\\
         q/cx
    \end{array};q,-\frac{a^2y}{q}\right).    
\end{align*}
\end{proof}

\subsection{Summations involving Bailey-Daum summation formula}

The Bailey-Daum summation formula is
\begin{equation}
    {}_2\phi_{1}\left(\begin{array}{c}
         a,b\\
         aq/b
    \end{array};q,-q/b\right)=\frac{(-q;q)_{\infty}(aq,aq^2/b^2;q^2)_{\infty}}{(aq/b,-q/b;q)_{\infty}}.    
\end{equation}

\begin{corollary}
    \begin{multline}
        {}_{2}\psi_{2}\left[
    \begin{array}{c}
         a,-1/dx^2\\
         d^2x,-a/dx^2
    \end{array}
    ;q,dx
    \right]\\
    =\frac{(-1/dx^2,q,d^2x/a,adx,q/adx,-q;q)_{\infty}(aq,aq^2/d^2;q^2)_{\infty}}{(-a/dx^2,d^2x,q/a,dx,d/a,aq/d,-q/d;q)_{\infty}}.
    \end{multline}
\end{corollary}
\begin{proof}
Set $b=d^2x$ and $y=-1/dx$ in Theorem 29.
\end{proof}

\begin{corollary}
    \begin{multline}
        {}_{2}\psi_{2}\left[
    \begin{array}{c}
         a,-q^2/acx\\
         qc^2x,-qc
    \end{array}
    ;q,\frac{qcx}{a}
    \right]\\
    =\frac{(-1/c,q,qc^2x/a,acx,q/acx,-q;q)_{\infty}(q^2/cx,aq/c^2x;q^2)_{\infty}}{(-q/acx,qc^2x,q/a,cx,qc/a,q/cx,-a/c;q)_{\infty}}.
    \end{multline}
\end{corollary}
\begin{proof}
Set $b=qc^2x$ and $y=q/ac$ in Theorem 30.
\end{proof}


\end{document}